\documentclass[10pt]{amsart}
\usepackage{amscd,amssymb}
\usepackage{pb-diagram}
\theoremstyle{plain}
\newtheorem{Conj}[equation]{Conjecture}
\newtheorem{Thm}[equation]{Theorem}
\newtheorem{Cor}[equation]{Corollary}
\newtheorem{Prop}[equation]{Proposition}
\newtheorem{Lem}[equation]{Lemma}
\newtheorem{Remark}[equation]{Remark}
\newtheorem{Def}[equation]{Definition}
\numberwithin{equation}{section}

\newcommand{\tr}{\operatorname{tr}}
\newcommand{\Nm}{\operatorname{Nm}}
\newcommand{\Ker}{\operatorname{Ker}}
\newcommand{\Image}{\operatorname{Im}}
\newcommand{\Ind}{\operatorname{Ind}}
\newcommand{\ind}{\operatorname{ind}}
\newcommand{\Hom}{\operatorname{Hom}}

\newcommand{\Aut}{\operatorname{Aut}}
\newcommand{\End}{\operatorname{End}}

\newcommand{\G}{\mathbb{G}}
\newcommand{\I}{\mathbb{I}}

\newcommand{\sgn}{\operatorname{sgn}}

\newcommand{\disc}{\operatorname{disc}}
\newcommand{\diag}{\operatorname{diag}}
\newcommand{\C}{\mathbb C}
\newcommand{\A}{\mathbb{A}}

\newcommand{\Z}{\mathbb{Z}}
\newcommand{\R}{\mathbb{R}}

\newcommand{\ol}{\overline}

\newcommand {\integral}[1]{\int\limits_{{#1}(F)\backslash {#1}(\A)}}
\newcommand{\bm}{\begin{multline*}}
\newcommand{\tu}{\end  {multline*}}

\title{ The non-tempered $\theta_{10}$  Arthur parameter
and Gross-Prasad Conjectures}
\author{Nadya Gurevich and Dani Szpruch}
\begin{document}
\maketitle

\begin{abstract}
We provide a construction of
local and automorphic non-tempered Arthur packets $A_\Psi$
of  the group  $SO(3,2)$  and its inner form $SO(4,1)$ associated
with Arthur's parameter
$$\Psi:\mathcal L_F\times SL_2(\C)\rightarrow O_2(\C)\times SL_2(\C)
\rightarrow Sp_4(\C)$$ and prove a multiplicity formula.

We further study the restriction of the representations
in $A_\Psi$ to the subgroup $SO(3,1).$
In particular, we discover that the local Gross-Prasad conjecture,
formulated for generic L-packets, does not generalize naively
to a non-generic A-packet.
We also study the non-vanishing of the
automorphic $SO(3,1)$-period on the group
$SO(4,1)\times SO(3,1)$ and $SO(3,2)\times SO(3,1)$ for
the representations above.

The main tool is the local and global theta correspondence for unitary
quaternionic similitude dual pairs.
\end{abstract}

\section{Introduction}
Let $F$ be a global  field and let $\mathcal L_F$
be its conjectural Langlands group. Recall that for every
place $v$ of $F$ there is an embedding
$i_v: W'_{F_v}\hookrightarrow \mathcal L_F,$
where $W'_{F_v}$ is the Weil-Deligne group of the local
field $F_v$.
In this paper we consider a non-tempered  Arthur parameter
$\Psi:\mathcal L_F \times SL_2(\C)\rightarrow Sp_4(\C)$
such that the image of a unipotent element of $SL_2(\C)$
belongs to the orbit generated by the short root element.
Such a parameter is called a parameter of $\theta_{10}$ type.
The global parameter $\Psi$ gives rise to
a family of local non-tempered parameters
$$\Psi_v=\Psi\circ i_v: W'_{F_v}\times SL_2(\C)\rightarrow Sp_4(\C).$$

Let $(V,q^\pm)$  be a non-degenerate quadratic space over
a non-archimedean field $F_v$ of
dimension $5$,  discriminant $1$
and the normalized Hasse invariant $\pm 1$.
The groups $SO(V,q^\pm)$ are pure inner forms of each other and share
the same dual group $Sp_4(\C)$.
According to Arthur's conjecture
there exist finite sets  $A_{\Psi_v}$ of
admissible  unitary representations of $SO(V,q^\pm)(F_v)$
corresponding to the parameter $\Psi_v$, called local $A$-packets. The construction of Arthur packets is a deep question which is far from
being solved. For many small rank cases the packets are
constructed by  the theta correspondence method.

\subsection{Construction of $A$-packets}
In this paper we construct the local $A$-packets of $\theta_{10}$ type
for the split group $SO(V,q^+)$ and its inner form $SO(V,q^-)$
using the similitude theta correspondence for quaternionic unitary dual pairs.
The quaternionic unitary dual pairs are relevant to the problem
since the groups $SO(V,q^\pm)(F_v)$ are isomorphic to
the group of projective similitude automorphisms of the
two-dimensional Hermitian
space over $D$, where $D$ runs over the set of all quaternion algebras
over $F_v$.
When $F_v=\R$ the same similitude theta correspondence
allows to define the Arthur packets for the groups $SO(3,2)$
and $SO(4,1)$, but not for the anisotropic form $SO(5)$.

We justify our construction using  global methods.
Let $(V,q)$ be a non-degenerate quadratic space of dimension $5$ over
a number field $F$ that is not anisotropic over any real place.
Taking tensor products of the representations of $SO(V,q)(F_v)$
in the local $A$-packets one can form a set of nearly equivalent
representations of $SO(V,q)(\A)$, where $\A$ is the ring of adeles of $F$.
This set is called the  global A-packet.
For any representation in the global A-packet Arthur predicts
a formula for its multiplicity in
 the discrete automorphic $L^2$ spectrum of $SO(V,q)$.
We construct an automorphic realization of certain representations
in the global A-packets and prove this multiplicity formula.

As another justification we show that the constructed cuspidal
representations form a nearly equivalence class of cuspidal representations.
The proof of the latter statement involves an $L$-function argument.

The cuspidal automorphic representations of $\theta_{10}$ type
of the split group $Sp_4\simeq Spin_5$  were considered by Piatetskii-Shapiro and Howe [HPS]
using the theta correspondence for the dual pair $(O_2, Sp_4)$
as the first counterexamples to the Ramanujan conjecture.
Later, Soudry in [S] used the similitude global theta correspondence
to construct non-tempered CAP representations of
$GSp_4$ and  investigated their properties. In [Ya2]
Yasuda has  constructed the local Arthur packets of $\theta_{10}$ type
of the group $Sp_4$ and its inner form
using the theta correspondence for quaternionic dual pairs.
We modify his results for the similitude theta correspondence.
Note, that for the similitude group the multiplicity one property
holds for the  cuspidal representations in the packets of type $\theta_{10}$,
when the representations of $Sp_4$ constructed by Yasuda can have high mulitplicity
in the discrete spectrum.


\subsection{The restriction problem}
Our second goal is to investigate the restriction problem over
a local non-archimedean field $F_v$.
Assume for now that $(V,q)$ is an arbitrary non-degenerate quadratic space
and $(U,q|_{U})$ is its non-degenerate subspace of codimension one.

The main object of the restriction problem is to compute,
for all irreducible representations $\Pi$ of $SO(V,q)$ and $\pi$ of
$SO(U,q|_U)$, the dimension of $$\Hom_{SO(U)}(\Pi,\pi).$$
The recent multiplicity one result in [AGSR] with a refinement in [W]
shows that the dimension of this space  is at most one.

About twenty years ago Gross and Prasad in [GP1] and [GP2]  have formulated a conjecture
according to which, given two  generic  local Langlands parameters
$$\Phi_1: W'_{F_v}\rightarrow {^L(SO(V,q))}(\C),\quad
\Phi_2: W'_{F_v}\rightarrow {^L(SO(U,q|_U))}(\C),$$
there exists a unique quadratic space $(V',q')$  with
a non-degenerate subspace $(U',q'|_{U'})$ of codimension one such that
$$\dim V' =\dim V, \quad \disc(V',q')=\disc(V,q), \qquad
\dim U' =\dim U,\quad \disc(U',q'|_{U'})=\disc(U,q|_U),$$
and unique representations
$\Pi$ of $SO(V')$ and $\pi$ of $SO(U')$ in the local Langlands
packets associated with the parameters $\Phi_1$ and $\Phi_2$ respectively  such that
$$\Hom_{SO(U')}(\Pi,\pi)\neq 0.$$
Equivalently,
\begin{equation}\label{GP:conj}
\sum_{V'\subset U'}\,\, \sum_{\Pi,\pi} \dim \Hom_{SO(U')}(\Pi,\pi)=1.
\end{equation}
Here $SO(V')\supset SO(U')$ runs over all relevant pure inner forms of
$SO(V)\supset SO(U)$ and  $\Pi$ and $\pi$ run
over representations in the Langlands packets associated with $\Phi_1,\Phi_2$ respectively.

The conjecture has first been proven in several
low rank cases (see [P1],[P2])  and later proven in
its full generality by Moeglin and Waldspurger in a series
of papers (see [MW])
assuming some natural properties of the generic Langlands packets.
It has also recently been  revised and generalized by
Gan, Gross and Prasad in [GGP] for all the classical groups.

It is natural to ask what happens if the parameters
are not generic. Having global applications in mind it is natural to replace
the Langlands parameters $\Phi_j$ and the associated Langlands packets   by
the Arthur's parameters $\Psi_j$ and  the Arthur's packets $A_{\Psi_j}$
associated with them. By Shahidi's conjecture every tempered
Arthur parameter is generic Langlands parameter.
Thus we concentrate on the non-tempered Arthur's parameters.

At the moment there is no satisfactory generalization of the
Gross-Prasad conjecture for the non-tempered Arthur packets.
However, examining the small rank cases for which the
construction of the Arthur packets is known  we quickly see
that the picture turns out  to be quite different.
In particular,  the sum (\ref{GP:conj}) is not always positive.
In the cases it is positive, we call the pair of parameters
$(\Psi_1,\Psi_2)$  {\sl locally admissible.}
Furthermore,  for locally admissible pairs the sum (\ref{GP:conj}) can be
greater than one.

Let us elaborate on the picture for the case  $\dim V=5$ and
$\disc(q)=1$. In this case $U$ is a non-degenerate subspace of $V$ of dimension $4$
and  $\disc q|_U=d$. The discriminant algebra $K$ of $q|_U$ is defined by
$$K=\left\{
\begin{array}{ll}
F[\sqrt{ d}] & d \notin (F^\times) ^2\\
F\times F &  d \in (F^\times)^2
\end{array}\right..$$
By abuse of notations we shall write $\disc(q|_U)=K$.

The non-tempered parameters $\Psi_1$ are partitioned into families
according to the orbit
of the image of $SL_2(\C)$, or, by the Jacobson-Morozov theorem,
according to a non-trivial unipotent orbit of $Sp_4(\C)$.
There exist  three families of  non-tempered Arthur parameters
corresponding to the three non-trivial unipotent orbits of
$^LSO(V)\simeq Sp_4(\C)$: the regular one,
the one generated by a long root
and the one generated by a short root of $Sp_4(\C)$.

When $\Psi_1$ is associated with the regular orbit,
the packet is a singleton and  consists of a one-dimensional representation.
Thus, $(\Psi_1,\Psi_2)$ is admissible only for non-tempered $\Psi_2$
 corresponding  to the regular orbit of $^L(SO(U))=SO_4(\C)$
and the restriction question is trivial.

The parameters  $\Psi_1$  associated with the  long root orbit
are called parameters of Saito-Kurokawa type.
Arthur's packets for all the inner forms
were constructed in [G] and the restriction problem has been considered
in [GG]. In particular, the condition for a pair
$(\Psi_1,\Psi_2)$ to be admissible was determined. Note that
$\Psi_2$ must be tempered. The sum (\ref{GP:conj})
can equal $1$ or $2$. However, for a fixed space $V'$
we have
\begin{equation}\label{GP:conj:group:fixed}
\sum_{\Pi,\pi} \dim \Hom_{SO(U')}(\Pi,\pi) \leq 1.
\end{equation}
Here the representations $\Pi$ of $SO(V')$  and $\pi$  of $SO(U')$
run over the representations in the packets associated with
the parameters $\Psi_1$ and $\Psi_2$ respectively.

Finally, for $\Psi_1$  associated with the short root orbit, i.e.,
of $\theta_{10}$ type, we solve the restriction problem in this paper.
Assuming $\disc(q|_U)$ is a field we determine the parameters
$\Psi_2$ such that $(\Psi_1,\Psi_2)$ is locally admissible.
Similar to  Saito-Kurokawa case, the parameter  $\Psi_2$ must be tempered.
We compute the value of (\ref{GP:conj}): it can be either $2$ or $4$.
Furthermore, even for a fixed form $V'$ the sum (\ref{GP:conj:group:fixed})
can be bigger than one. This phenomenon has not occurred before.
The local restriction theorem appears in Section \ref{local:restriction}.

For the case $\disc(q|_U)$ is a split quadratic algebra
we obtain the restriction of representations in $A_{\Psi_1}$
to the split group $SO(2,2)$, but not to its anisotropic inner form $SO(4)$.

\begin{Remark} \label{assumption}
The difficulty that prevents us  from
computing the restriction to the anisotropic group
is of technical nature.
Our main tool is the see-saw duality for
a pair of similitude quaternionic unique dual pairs.
However there is a difficulty to define
the theta correspondence for such {\sl similitude}
dual pairs when neither one of the quaternionic
Hermitian spaces admits a polarization as a sum of two
isotropic subspaces. Hence, the construction will not work
whenever the group $SO(U)$ is anisotropic.
The same difficulty prevents one to construct representations
in the Arthur packet of the anisotropic group $SO(5)$ over the field of real numbers.

Including the anisotropic case would bring a lot of new notations
and discussions which we feel do not belong here.
We shall treat the remaining case elsewhere.
\end{Remark}

The local restriction problem  has a global counterpart.
For any $(V',U')$ as above, define a functional
$P_{V',U'}$ on the space of automorphic forms
$\mathcal A(SO(V'))\otimes  \ol{\mathcal A_{cusp}(SO(U'))}$
by
$$P_{V',U'}(F,f)=
\int\limits_{SO(U')(F)\backslash SO(U')(\A)} F(h) \ol{f(h)} \, dh.$$

Given an automorphic representation $\Pi$  of $SO(V')(\A)$
and a cuspidal automorphic representation $\pi$  of $SO(U')(\A)$,
we investigate the non-vanishing of $P_{V',U'}$ on $\Pi\boxtimes \ol\pi$.
Two  parameters $\Psi_1$ and $\Psi_2$ such that $P_{V',U'}$  is non-trivial
on some  representation  $\Pi\boxtimes \ol\pi$ of $SO(V')\times SO(U')$
from the global packet $A_{\Psi_1}\times A_{\Psi_2}$
are called {\sl globally admissible.}

Let   $\Psi_1$ be a parameter of the type $\theta_{10}$, and
$E$ be a quadratic field extension naturally associated with it.
Let $\Psi_2$ be a tempered parameter of $SO(U, q|_U)$.

Our main global theorem states:

\begin{Thm}\label{global:condition}
Let $\Pi$ be an automorphic representation of $SO(V')(\A)$ in $A_{\Psi_1}$ and
let $\pi$ be a cuspidal representation of $SO(U')(\A)$. The following statements
are equivalent.
\begin{enumerate}
\item
The period $P_{V',U'}$ does not vanish on the $\Pi\boxtimes \ol{\pi}$.
\item
$\Hom_{SO(U')(\A)}(\Pi,\pi)\neq 0$ and  $E\neq \disc(q|_{U'})$.
\end{enumerate}
\end{Thm}

This agrees with a version of the refined Ichino-Ikeda conjecture.
For tempered cuspidal representations $\Pi$ of $SO(V')(\A)$
and $\pi$ of $SO(U')(\A)$  Ichino and Ikeda in [II] have conjectured
that the period $P_{V',U'}$ does not vanish on $\Pi\boxtimes \ol{\pi}$
if and only if
\begin{equation}\label{II:global:condtition}
\Hom_{SO(U')(\A)}(\Pi,\pi)\neq 0,\quad \quad
\frac{L(\Pi\boxtimes \ol{\pi},1/2)}{L(\Pi\boxtimes\ol{\pi}, Ad,1)}\neq 0.
\end{equation}

We show that the second statement of Theorem \ref{global:condition}
is equivalent to (\ref{II:global:condtition}) .

\begin{Remark}
For the same reasons as before we investigate the non-vanishing of
 $P_{V',U'}$ assuming
\begin{itemize}
\item $(V',q')$ is not anisotropic over any real place.
\item $(U',q'|_{U'})$  is not anisotropic for any local place.
\end{itemize}
\end{Remark}

The paper is organized as follows: after explaining in Section $2$ some
generalities about (skew)-Hermitian spaces over division algebras,
we introduce in the Sections $3$ and $4$
the group of similitude automorphisms of these spaces.

In Section $5$ we recall the notion of Howe duality for  similitude unitary  quaternionic
groups. This is  our main tool for constructing the A-packets.
In Section $6$ we determine the theta correspondence as explicitly as possible.
The local A-packets are defined  in the Section $7$.
The restriction problem is solved in the Section $8$ using the see-saw duality.
The rest of the paper addresses global questions.
In  section $9$ we define the global Arthur packet and compute
the multiplicity predicted by Arthur's multiplicity formula.
The automorphic realization
of the global Arthur packets is obtained in Sections
$10$ and $11$ using the global theta correspondence.
The Arthur's multiplicity formula  is established in the Section
$12$.
Section $13$ is devoted to the Rankin-Selberg integral representation
of an L-function of degree $5$ of a representation of $SO(V')$.
When the group $SO(V')$ is split,  this integral representation
has been constructed by  Piatetskii-Shapiro and Rallis.
This L-function is used  in Section $14$ to show
that the  cuspidal representations that we have constructed
constitute a full nearly equivalence
class, so that our construction is exhaustive.
The global restriction problem
is solved using the global see-saw duality in the Sections $15-17$.
The main global theorem is (\ref{global:restriction:thm}).
Finally, in Section $18$ we show the compatibility
of the obtained results with the Ichino-Ikeda conjecture.

\subsection{Acknowledgement}
The first author is partially supported by ISF grant $1691/10$.
We thank W.T. Gan for his help and attention.

\section{ Hermitian and skew-Hermitian spaces}

Let  $F$ be a local non-archimedean
field. Let $D$ be a (possibly split) quaternion algebra over $F$. Denote by $\sigma$ the main involution on $D$
and by $\Nm_{D/F}$ the reduced norm. To any right (left) $D$-module $M$ we associate a left (right)
$D$ module $\ol{M}$ by
$$\ol{M}=\{\bar m : m\in M\}, \quad
\ol{m_1}+\ol{m_2}=\ol{m_1+m_2},\quad
 d\cdot \ol {m}=  \ol{m \cdot\sigma(d)}\quad
(\ol{m} \cdot d= \ol{\sigma(d)\cdot m}) .$$

We define a right skew-Hermitian space $(V,s)$ to be a right
free $D$-module together with a map
$s: \ol V\times V\rightarrow D$ such that
$$s(d_1 \ol{v_1}, v_2 d_2)=  d_1 s(\ol{v_1}, v_2) d_2,
\quad s( \ol{v_2},  v_1)= - \sigma(s(\ol{v_1},v_2)).$$
All the non-degenerate skew-Hermitian free modules of rank $n$ over $D$
are classified up to isometry by the discriminant in $F^\times/(F^\times)^2$,
or equivalently by quadratic algebras over $F$.

\vskip 10pt

We also define a left Hermitian space $(W,h)$ to be a left free
$D$-module $W$ together with the  map $h:W \times \ol{W}\rightarrow D$
such that
$$h(d_1 w_1, \ol{w_2} d_2)= d_1 h(w_1,\ol{w_2}) d_2,\quad
h( w_2,\ol{w_1})= \sigma( h( w_1,\ol{w_2})).$$

For any even $n$ there exists a unique, up to isometry,
non-degenerate left Hermitian space
of rank $n$ over $D$.


\vskip 10pt
\subsection{Morita equivalence}
Let $D$ be a split algebra. Equivalently,
$D=\End_F(M)$ for
a two-dimensional space $M$ over $F$,
viewed as a right $D$ module. The space $M\otimes_D \ol{M}$ is one dimensional over $F$.
Fix an isomorphism
$\varphi_M: M\otimes_D \ol{M}\rightarrow F$.
The choice of $\varphi_M$ fixes an isomorphism $\ol{M}\simeq M^\ast$.
In particular, there is an isomorphism
$\ol{M}\otimes_F M\simeq \End_F(M)=D$.

\vskip 10pt

For any left Hermitian space $(W_D, h_D)$ over $D$
there corresponds a symplectic
space $(W,h)$ over $F$ defined by
$$W= M\otimes_D W_D, \quad
h(m_1\otimes w_1, m_2\otimes w_2)=
\varphi_M(m_1 h_D(w_1,\ol{w_2})\otimes \ol{m_2}).$$

\vskip 10pt

Similarly, for any right skew-Hermitian space $(V_D,s_D)$
there corresponds a quadratic space $(V,s)$ over $F$ defined by
$$V= V_D\otimes_D \ol{M}, \quad
s(v_1\otimes \ol{m_1}, v_2\otimes \ol{m_2})=
\varphi_M(m_1 s_D(\ol{v_1}, v_2) \otimes \ol{m_2}).$$

Obviously,

$$\dim_F V= 2\dim_D V_D, \quad  \dim_F W =2\dim_D W_D, \quad
\disc(V)=\disc(V_D). $$

Note that the isomorphism $\varphi_M$ is not canonical.
For an element $a\in F^\times \backslash \Nm_{E/F}(E^\times)$ the
isomorphism $\varphi_M$ and $a\varphi_M$ give rise
to two quadratic spaces $(V,s^\pm)$ having the  same discriminant
but different normalized Hasse invariants. Hasse invariant $h(V,q)$ for
a quadratic space is normalized so that  it is constant in any
Witt tower.

Moreover, there are isomorphism of $D$-modules
$$V\otimes_F M= V_D\otimes_D \ol{M}\otimes_F M\simeq V_D,
\quad \ol{M}\otimes_F W =\ol{M}\otimes_F M\otimes_D W_D\simeq W_D.$$

\subsection{Main players}
Let  us fix the Hermitian and skew-Hermitian
spaces that will be considered in the paper.
  First we fix notations for the algebras.

\begin{itemize}
\item Let $D$ be a  quaternion algebra over $F$ with the main involution
$\sigma$. Define $h(D)=1$ if $D$ split and $h(D)=-1$ otherwise.
\item Let $E$ be a  quadratic algebra over $F$ contained in $D$.
The involution $\sigma$ restricted to $E$ defines a non-trivial
Galois action.
\item Let $K$ be a  quadratic algebra over $F$.
\item Denote  $L=E\otimes_F K$. It is a  quadratic algebra over $K$.
\item  Denote $D_K=D\otimes_F K $. It is a quaternion algebra over $K$.
The main involution is still denoted by $\sigma$ and $\Nm_{D_K/K}$ denotes the reduced norm.
\end{itemize}

\begin{enumerate}
\item Let $(V_D,s_D)$ be the one-dimensional
right skew-Hermitian space over $D$ of discriminant $E$.

\item Let $(V_{D_K}=V_D \otimes_F K, s_{D_K})$ be the one-dimensional
right skew-Hermitian space over $D_K$,
where $$s_{D_K}(\ol{v_1\otimes k_1},v_2\otimes k_2)=s_D(\ol{v_1},v_2)\ol{k_1}k_2.$$
It has  discriminant $L=E\otimes_F K$.
\vskip 5pt
\item Let $(W_{D_K}, h_{D_K})$ be the one-dimensional
left Hermitian space over $D_K$.
\vskip 5pt
\item Let $(W_D=R_{K/F} W_{D_K}, h_D)$ be the two-dimensional
right Hermitian space over $D$ obtained from
$W_{D_K}$ by a restriction of scalars. The form $h_D$ is defined by
$$h_D(w_1,\ol{w_2})=\tr_{D_K/D} h_{D_K}(w_1,\ol{w_2}).$$
\end{enumerate}

Assume that the algebra $D_K$ splits.

\begin{enumerate}
\item Denote by $(V_K, s^\pm_K)$ the two-dimensional quadratic  spaces over $K$
of discriminant $L$, Morita equivalent to $(V_{D_K}, s_{D_K})$.

\item Denote by $(W_K,h_K)$
the two-dimensional symplectic  space over $K$,
Morita equivalent to $(W_{D_K}, h_{D_K})$.
\end{enumerate}

Assume that the algebra $D$ splits.
\begin{enumerate}
\item Denote by $(V_F, s^\pm_F)$ the two-dimensional quadratic  spaces over $F$
of discriminant $E$, Morita equivalent to $(V_D, s_D)$.

\item Denote by $(W_F,h_F)$
the four-dimensional symplectic  space over $F$,
Morita equivalent to $(W_D, h_D)$.
\end{enumerate}

Note the obvious relations:
 $$V_K=V_F\otimes_F K, \quad W_F=R_{K/F} W_K.$$

\subsection{Symplectic forms on tensor products}
The space $V_D\otimes_D W_D$ admits a symplectic
form defined by
$$\langle v\otimes w, v'\otimes w'\rangle=
\tr_{D/F}s_D(\ol{v},v')\sigma(h_D(w,\ol{w'})). \quad $$
Similarly, the space $V_F\otimes_F W_F$ admits a symplectic form defined by
$$\langle v\otimes w, v'\otimes w'\rangle=
s(v, v')h(w, w'). \quad $$
Finally, the space $V_K\otimes_K W_K$ is a symplectic space
over $K$. Composing the symplectic form  with $\tr_{K/F}$
we obtain the symplectic form over $F$.

\begin{Lem}
\begin{enumerate}
\item Suppose that $D$ splits. There is a natural isomorphism
of symplectic $F$ spaces
$$V_F\otimes_F W_F \simeq V_D\otimes_D W_D.$$
\item Suppose that $D_K$ splits. There is a natural isomorphisms of symplectic $F$ spaces
 $$V_{K}\otimes_{K} W_{K}\simeq V_D\otimes_D W_D.$$
\end{enumerate}
\end{Lem}

\begin{proof}
Since $D$ splits, $D=\End_F(M)$. To prove (1) note that
$$V_F\otimes W_F = V_D\otimes_D \ol{M}\otimes_F M\otimes_D W_D \simeq V_D\otimes_D W_D.$$

To prove (2) note that
$$V_K\otimes_K W_K \simeq V_{D_K}\otimes_{D_K} W_{D_K}$$
and so
$$V_K\otimes_K W_K \simeq V_D\otimes_D {D_K}\otimes W_{D_K} \simeq V_D\otimes_D W_D.$$
Clearly, all the natural isomorphisms above are isomorphisms of symplectic spaces.
\end{proof}

\subsection{Compatibility of polarizations}
If  $D_K$ splits there exists
a two-dimensional space $M_K$ over $K$ such that $D_K=\End_K(M_K)$.
Given a polarization of $W_K$ as a sum of two isotropic spaces
$$W_K=X_K\oplus Y_K, $$
the {\sl compatible polarization} of $W_D$ is defined by
$$W_D=X_D\oplus Y_D, \quad X_D=\ol {M_K} \otimes_K X_K,
 \quad Y_D=\ol{ M_K} \otimes_K Y_K.$$
Similarly if $D$ splits, the polarizations
$$W_F=X_F\oplus Y_F, \quad W_K=X_K\oplus Y_K$$
are called {\sl compatible} if
$$X_F=R_{K/F} X_K,\quad  Y_F=R_{K/F} Y_K.$$

\begin{Lem}\label{compatibility:polarizations}
\begin{enumerate}
\item For compatible polarizations there is a natural
isomorphism of $F$-spaces
$$V_F\otimes_F X_F \simeq V_K\otimes_K X_K.$$

\item Suppose that $D_K$ splits. For compatible polarizations there is a natural
isomorphism of $F$-spaces
$$V_D\otimes_D X_D \simeq V_K\otimes_K X_K.$$

\end{enumerate}
\end{Lem}

\begin{proof}
Part one is trivial. To prove Part (2) simply note that
$$V_D\otimes_D X_D\simeq
V_D\otimes_D D_K\otimes_{D_K} \ol{M_K}\otimes X_K \simeq
V_{D_K} \otimes_{D_K}  \ol{M_K}\otimes_{K} X_K\simeq   V_K\otimes_K X_K.$$
\end{proof}

The forms $h_D, h, h_K$ define natural isomorphisms
$$X_D\simeq \ol{Y_D}^\ast, \quad X_F\simeq Y_F^\ast,\quad
X_K\simeq Y_K^\ast.$$

\section{Hermitian Unitary Groups and their representations}\label{sec:groups}

Many groups will be used in the course of the paper. Let us introduce
some preliminary notations:
\begin{itemize}
\item For any algebraic group $M$ denote by $M^c$ its
connected component.
\item If $M$ is a subgroup of a group of similitude
of a Hermitian/skew-Hermitian space, denote by $M^1$ the subgroup
of elements of $M$ whose similitude is $1$.
\item For a group $M$ over a field $F$ and
an extension  $E$ over $F$ we denote by $R_{E/F}M$ the $F$-group
obtained from $M$ by restriction of scalars.

\item Let $M$ be an algebraic group over a local field $F$.
A representation $\pi$  of $M(F)$ (or just of $M$ if there is no confusion)
is a smooth representation if $F$ is non-archimedean and
smooth Frechet representation of moderate growth if $F$ is archimedean.

\item Let $N(F)$ be a normal subgroup of $M(F)$ and let
 $\pi$ be a representation of $N(F)$. For $m\in M(F)$.
Denote by $\pi^m$ the conjugate representation, i.e.,
$\pi^m(n)=\pi(mnm^{-1})$.
\end{itemize}

\subsection{ The group $G_D$}
Let $M_{2}$ be the $F$-group of $2\times 2$ matrices.
The $F$-group $G_D$ is defined by
 $$
G_D=\{ g\in  M_2(D): \exists \lambda(g)\in \G_m :
gJ\sigma(g)^t = \lambda(g)^{-1} J\},
$$
where $J=\left(\begin{smallmatrix} 0&1\\ 1 & 0\end {smallmatrix}\right)\in M_2$. 
The group $G_D$ acts on the algebraic Hermitian vector space $(W_D,h_D)$
by $g.w=w g^{-1}$ and it is isomorphic to the group of the
similitude automorphisms with the similitude character $\lambda$. That is,

$$G_D \simeq \{g\in Aut(W_D):
\forall w_1,w_2\in W_D \quad
h_D(g.w_1,g.w_2)=\lambda(g)h_D(w_1,w_2) \}.$$

We denote by $Z_D$ the center of $G_D$. Obviously $Z_D\simeq \G_m$.
\subsubsection{The parabolic subgroup $P_D$}

We fix a polarization $W_D=X_D\oplus Y_D$ and define
$P_D$ to be the subgroup of $G_D$ preserving the subspace $Y_D$. It is
a maximal parabolic subgroup. If $D$ does not split
then $P_D$ is unique parabolic subgroup of $G_D$.

One has a Levi decomposition $P_D=M_D \cdot U_D$, where
$M_D\simeq GL(X_D)\times \G_m$. We realize $M_D$ in $G_D$
as
$$M_D=\left\{m(a,t)=\left(\begin{array}{ll} ta &0 \\
 0& (\ol{a}^\ast)^{-1}
\end{array}\right), \quad
 a\in GL(X_D) , t\in \G_m\right\},$$
where $\bar a^\ast\in GL(Y_D)$ is characterized by
$$h_D(x,\ol{\ol{a}^\ast(y)})=h_D(a(x),\ol{y}).$$
In particular $\lambda(m(a,t))=t^{-1}$.

The unipotent radical $U_D$ can be identified with
$$\{S\in \Hom(X_D,Y_D), S=-\ol{ S}^\ast\}\simeq
\{{\rm skew-Hermitian \, forms\, on}\, X_D\},$$
where $\ol{S}^\ast \in \Hom(X_D,Y_D)=\Hom(X_D, \ol{X_D}^\ast)$
 is characterized by
$$h_D(y_1, \ol{S(y_2)})=h_D(\ol{S}^\ast(y_1),\ol{y_2}).$$
For any such $S$, the element $u(S)\in U_D$ acts trivially on $Y_D$
and for $x\in X_D$   one has $u(S)(x)=x+S(x)$.

\subsubsection{The action of $M_D$ on $U_D$}

The group $M_D$ acts on $U_D$ by conjugation
$$m(a,t). u(S)= u(t \ol{a}^\ast S a  ).$$

Note that two  Hermitian forms lie in the same orbit
if and only if  they have the same discriminants in $F^\times /(F^\times)^2$. Hence, these orbits are classified by the quadratic algebras inside $D$.

\subsubsection{ The characters of $U_D$}
Fix a non-trivial  additive character $\psi$ of $F$.
The group of unitary characters of $U_D$ can be identified
with  $\Hom(U_D,\G_a)$ via composition with $\psi$.
The group  $M_D$ acts on $\Hom(U_D,\G_a)$ by the dual action to
the one discussed above.

\begin{Prop}
\begin{enumerate}
\item
The set $\Hom(U_D,\G_a)$ is naturally identified with the set
of skew-Hermitian forms on $Y_D$.
\item
$M_D$-orbits on  $\Hom(U_D,\G_a)$ are parameterized
by quadratic algebras over $F$.
\item For any skew-Hermitian form $T$ on $Y_D$
 one has $Stab_{M_D}(\Psi_T)\simeq  GU(Y_D,T).$
\end{enumerate}
\end{Prop}
\begin{proof}

$(1)$. Indeed, given two maps
$T\in \Hom(Y_D, X_D)$ and $S\in \Hom(X_D,Y_D)$ such that
$T=-{\ol T}^\ast,S=-{\ol S}^\ast$ there is a non-degenerate pairing
 $$\langle T, S\rangle =\tr_{D/F} TS.$$
 In particular,
for $T\neq 0$ we define a non-trivial character
$$\Psi_{T}(u(S))=\psi\left(\tr_{D/F} TS\right).$$

$(2)$ follows from $(1)$.

$(3)$
$$m(a,t)\Psi_T(u(S))=\psi(\tr_{D/F} t^{-1} a^{-1} T (\ol a^\ast)^{-1}S).$$
Hence, $$m(a,t) \Psi_T=\Psi_T \Leftrightarrow aT\ol a^\ast= t^{-1}T.$$
\end{proof}

\subsubsection{Induced representations}

Any irreducible representation of $M_D$ has the form
$\tau\boxtimes \chi$ where $\tau$ is an irreducible
representation of $GL(X_D)$ and $\chi$ is a character of $\G_m.$
We denote the induced representation associated with $\tau\boxtimes \chi$ by $$I_{P_D}(\tau\boxtimes \chi)=\Ind_{P_D}^{G_D} \tau\boxtimes \chi.$$  If $\tau\boxtimes \chi$ is
a product of a tempered representation and a positive
character the representation
$I_{P_D}(\tau\boxtimes \chi)$ has a unique irreducible
quotient denoted by $J_{P_D}(\tau\boxtimes \chi)$.
Note that
$I_{P_D}(\tau\boxtimes \chi) \simeq (\chi\circ \lambda^{-1})\otimes  I_{P_D} (\tau\boxtimes 1)$.

\subsection{The group $G_F$}
The group $G_F=GSp_4$, acting  on $W_F$ by $g.w=wg^{-1}$,
is the group of similitude automorphisms of the symplectic
space $(W_F,h_F)$.
The group $G_D$ is an inner form of $G_F$ and is isomorphic to it if and only if  $D$ splits.
 We realize $G_F$ as a group of matrices
$$\{g\in GL(W_F): g J_4 {g^t}=\lambda(g)^{-1} J_4 \} \quad
 J_4=
\left(\begin{array}{llll}
& & &1 \\
&& 1& \\
&-1 & &\\
-1 & & &
\end{array}
\right),
$$

The group $G_D$ is an inner form of $G_F$ and is isomorphic
to it if and only if  $D$ splits.

\subsubsection{Parabolic subgroups of $G_F$}

Fix a polarization $W_F=X_F\oplus Y_F$. The  maximal Siegel parabolic $P_F$
consists of the elements stabilizing the space $Y_F$.
One has a Levi decomposition $P_F=M_F\cdot U_F$ with $M_F\simeq GL_2\times \G_m$.

The other  maximal parabolic subgroup of $G_F$ is the Heisenberg parabolic subgroup.
It is denoted by $Q_F=L_F \cdot V_F$ or just by $Q$.
The Levi subgroup $L_F\simeq GL_2\times \G_m$
is embedded in $G_F$ as
$$L_F=\left\{l(t,g)=\left(\begin{array}{lll}
t & & \\
&g& \\
& & t^{-1} \det(g)
\end{array}
\right)
\quad
g\in GL_2, t\in \G_m
\right\}.
$$

We also fix a  Borel subgroup $B_F=Q_F\cap P_F$ with a Levi decomposition
$B_F=T_F N_F$.  The split torus  $T_F$ is realized inside $G_F$ as
$$T_F=\{t(a,b,s)=\diag(a,b,b^{-1}s, a^{-1}s)\}.$$
The similitude factor is given by
 $\lambda(t(a,b,s))=s^{-1}$.

\subsubsection{Induced representations}

Any irreducible representation of $L_F$
has the form $\chi\boxtimes \tau$  where $\tau$ is an irreducible
representation of $GL_2$ and $\chi$ is a character of $\G_m.$ If $\chi\boxtimes \tau$ is a product of a tempered representation and
a positive character then the representation
$\Ind^{G_F}_{Q_F} (\chi\boxtimes \tau)$
has a unique irreducible quotient denoted by $J_{Q}(\chi\boxtimes \tau).$

In the course of determining the theta correspondence in Section 6 we shall
need the decomposition of the representation
$I_{Q}(\mu):=\Ind^{G_F}_{Q_F} \mu^{-1}\boxtimes \mu \circ \det$.

\begin{Lem}([K] Theorem 4.2)\label{structure:of:induced}
\begin{enumerate}
\item For $\eta\neq 1,|\cdot|^{\pm 2}$ the representation $I_{Q_D}(\mu)$
is irreducible.
\item $I_{Q}(1)=J_{Q}(|\cdot|\boxtimes |\det|^{-1/2}\Ind^{GL_2}_B 1)\oplus
J_{P}(|\det|^{1/2} St\boxtimes |\cdot|^{-1/2})$.
\item The length of $I_{Q}(|\cdot|^{-2})$ is two.
It has $|\cdot|\circ \lambda$ as a unique quotient.
\end{enumerate}
Here $St$ denotes the Steinberg representation of $GL_2$.
\end{Lem}

For a character $\chi(t(a,b,s))=\chi_1(a)\chi_2(b)\chi(s)$ denote the
associated  induced representation  by $I_{B}(\chi_1,\chi_2; \chi)$.

\subsection{The groups $G_{D_K}, G_K$ and $G^0_{D_K}, G^0_K$}

The group $G_{D_K}$ over $K$ is the group of similitude
automorphisms of the one-dimensional
Hermitian space $W_{D_K}$ acting by  $g.w=w g^{-1}$
with the similitude character $\lambda_K$, i.e.,
$$G_{D_K}=\{g\in Aut_{D_K}(W_{D_K}):
\forall w_1,w_2\in W_{D_K} \quad
h_{D_K}(g.w_1,g.w_2)=\lambda_{D_K}(g)h_{D_K}(w_1,w_2)\}.$$

Since the space $W_{D_K}$ is one dimensional over $D_K$ one has
$$G_{D_K}(K)\simeq D_K^\times, \quad \lambda_{D_K}(g)=(\Nm_{D_K/K}(g))^{-1}.$$

The group $G_K$ over $K$ is the group of similitude
automorphisms of the two-dimensional
Hermitian space $W_K$ acting by  $g.w=w g^{-1}$
with the similitude character $\lambda_K$, i.e.,
$$G_K=\{g\in Aut_{K}(W_K):
\forall w_1,w_2\in W_{D_K} \quad
h_K(g.w_1,g.w_2)=\lambda_K(g)h_K(w_1,w_2)\}.$$
Since $h_K$ is a symplectic form it follows that
$$G_K\simeq GL_2,\quad  \lambda_K=(\det)^{-1}.$$
Fixing a polarization $W_K=X_K\oplus Y_K$ fixes a  Borel subgroup
$B_K$ that preserves $Y_K$. There is a Levi decomposition
$B_K=T_K\cdot N_K$.

The group $G_{D_K}$ is an inner form of $G_K$ and  is  isomorphic to $G_K$
if and only if  the algebra $D_K$ splits over $K$.

The similitude factor $\lambda_K$ induces the map of $F$-groups:
$$R_{K/F}\lambda_{D_K}: R_{K/F} G_{D_K}\rightarrow R_{K/F}\G_m,
\quad R_{K/F}\lambda_{K}: R_{K/F} G_{K}\rightarrow R_{K/F}\G_m$$
and there is a  canonical embedding $\G_m\hookrightarrow R_{K/F}\G_m$.

We define the following algebraic groups over $F$.
$$G^0_{D_K}= \{g\in R_{K/F} G_{D_K}: R_{K/F}\lambda_K(g)\in \G_m\}.$$
$$G^0_{K}= \{g\in R_{K/F} G_{K}: R_{K/F}\lambda_K(g)\in \G_m\}.$$

In particular
$$G^0_{D_K}(F)= \{g\in G_{D_K}(K): \lambda_K(g)\in F^\times\}.$$

There are natural inclusions of $F$-groups

$$G^0_{D_K}\hookrightarrow G_D,\quad  G^0_K\hookrightarrow G_F.$$
Note that under this embedding  $B^0_K\hookrightarrow P_F$.


The group $\G_m$  acts on $Res_{K/F}V_{D_K}$ and $Res_{K/F}V_K$ by scalar multiplication.
Denote its image in $G^0_K$ and $G^0_{D_K}$ by $Z^0_K$ and $Z^0_{D_K}$
respectively.  The groups $Z^0_K$  and $Z^0_{D_K}$ are contained in the center of
$G^0_K$ and $G^0_{D_K}$ respectively.


\subsubsection{Dihedral representations}

Let $\eta_L$ be a character of $L^\times$.
Denote by
$\pi(\eta_L)$ the dihedral representation of $G_K$
associated to $\eta_L$. If $L$ is a field and $\eta_L\neq \eta_L^\sigma$ then
$\pi(\eta_L)$ is supercuspidal.

\subsection{The unitary quaternionic groups  as special orthogonal groups}

The groups $G_D, G_{D_K}^0$ play main role in this paper because of
 the following accidental isomorphisms

\begin{Prop}
Let $(V,q)$ be a 5-dimensional quadratic spaces over $F$
with discriminant $1$.
Let $(U,q_K)$ be a non-degenerate quadratic subspace of $V$ of codimension $1$
and discriminant algebra $K$.

\begin{enumerate}
\item
$$G_D/Z_D\simeq SO(V,q), \quad h(V,q)=h(D).$$

\item
$$G^0_{D_K}/Z^0_{D_K}\simeq SO(U,q_K),\quad  h(U,q_K)=h(D_K).$$
\end{enumerate}

In particular the groups $G_D/Z_D$  and $G^0_{D_K}/Z^0_{D_K}$ over
a local non-archimedean field vary over all pure inner forms of $SO(V)$
and $SO(U)$ as $D$ varies over the set of quaternion algebras.
Note that $SO(U,q_K)$ is anisotropic if and only if the algebra $D_K$ splits.
\end{Prop}

\begin{proof}

We shall construct the isomorphisms explicitly.
One has $D=Span_F\{i,j,k\},$ where
$$i^2=\alpha, j^2=\beta, k=ij=-ji, \quad \alpha,\beta \in F^\times.$$
It is known that $h(D)=(\alpha,\beta)_F$, where $(\cdot, \cdot)_F$
denotes the Hilbert symbol.

Consider the space

$$M_2(D)\supset X=\{x \in M_2(D):  J x^t J^{-1}=\sigma(x)\}.$$
The space $X$ has dimension $6$ over $F$ and admits the symmetric form
$q(x,y)= \tr_{D/F}(\tr(xy))$.
The group $G_D$ acts on $X$ by conjugation. In particular
$Z_D$ acts trivially. This action  preserves the form $q$
and fixes the identity matrix.
Denote by $V$ the orthogonal complement to the identity matrix.
More precisely,
$$V=
\left\{
\left(
\begin{smallmatrix} a&b\\ c & -a\end{smallmatrix}\right)
: b,c \in F, a\in D, \tr_{D/F}(a)=0 \right\}$$
We obtain an isomorphism between $G_D/Z_D$ and $SO(V,q)$.
Considering the orthogonal basis of $V$
$$\left\{\left(\begin{smallmatrix} i&0\\ 0&-i\end {smallmatrix}\right),
 \left(\begin{smallmatrix} j&0\\ 0&-j\end {smallmatrix}\right),
 \left(\begin{smallmatrix} k&0\\ 0&-k\end {smallmatrix}\right),
 \left(\begin{smallmatrix} 0&1\\ -1 & 0\end {smallmatrix}\right),
 \left(\begin{smallmatrix} 0&1\\ 1 & 0\end {smallmatrix}\right)\right\}$$
it is easy to see to see that $\disc(V,q)=1$
and the normalized Hasse invariant $h(V,q)$ equals $h(D)$.
This proves $(1)$.

For an quadratic algebra $K=F[\sqrt d]$ consider a vector
$v_K=\left(\begin{smallmatrix} 0&1\\
d & 0\end {smallmatrix}\right)$$ \in V.$
In particular  $q(v_K,v_K)=d$.
The stabilizer of $v_K$ in $G_D$ isomorphic to $G^0_{D_K}$, i.e.,
$$\{g\in G_D: gv_Kg^{-1}=v_K\}\simeq G^0_{D_K}.$$
Denote by $U$ the orthogonal complement to the one-dimensional
space $\langle v_K\rangle$ and denote by $q_K$ the restriction
of the form $q$ to $U$. Then, $G^0_{D_K}$ acts irreducibly on $U$
and the subgroup  $Z^0_{D_K}$ acts trivially.
This defines an isomorphism
$$G^0_{D_K}/Z^0_{D_K}\simeq SO(U,q_K).$$

Considering the orthogonal basis of $U$
$$\left\{
\left(\begin{smallmatrix} i&0\\ 0&-i\end{smallmatrix}\right),
\left(\begin{smallmatrix} j&0\\ 0&-j \end{smallmatrix}\right),
\left(\begin{smallmatrix} k&0\\ 0&-k \end{smallmatrix}\right),
\left(\begin{smallmatrix} 0&1\\ \disc(q_K) & 0\end{smallmatrix}\right)
\right\}$$
it is easy to see that $\disc(U,q_K)=K$ and $h(U,q_K)=h(D_K)$.
This proves $(2).$
\end{proof}

Thus, instead of considering restrictions of representations of
$SO(V')$ to $SO(U')$ we consider the equivalent  question
of restriction of representations of $G_D$ with trivial central
character  to $G^0_{D_K}/Z^0_{D_K}$.

\begin{Remark}
When $F=\R$  the non-degenerate Hermitian forms over $D$
are classified by their signature and we have
$$G_D/Z_D\simeq \left\{
\begin{array}{ll}
SO(3,2) &  h(D)=1, \quad sig(W_D)=(1,1)\\
SO(4,1) & h(D)=-1, \quad sig(W_D)=(1,1)\\
SO(5,0)& h(D)=-1, \quad sig(W_D)=(2,0)
\end{array}
\right..
$$
Note that the  space $W_D$ with the signature $(2,0)$ is not hyperbolic.
\end{Remark}

\section{ The skew-Hermitian unitary groups}

We shall now define and discuss some auxiliary groups that will be used to construct the packets
of representations for the groups $G_D$ and $G^0_{D_K}$
and to determine the restrictions.
\subsection{The group $H_D$, $H_F$ and their representations}
Define $H_D$ to be the group of similitude automorphisms of $(V_D,s_D).$
It is a disconnected algebraic group. Denote its connected component
by $H_D^c$. Its properties are described below.

\begin{Prop}
\begin{enumerate}
\item One has  $H_D^c=R_{E/F} \G_m$
and the group $H_D$ fits into the exact sequence
$$1\rightarrow H_D^c\rightarrow H_D\rightarrow \mu_2\rightarrow 1. $$
This sequence splits if and only if $D$ splits.
\item
 The sequence of $F$-points is also exact
$$1\rightarrow E^\times \rightarrow H_D(F)\rightarrow \mu_2(F) \rightarrow 1. $$
In particular, both connected components of $H_D$ have points over $F$.

\item
$$\lambda(H_D(F))=\left\{
\begin{array}{ll} \Nm(E^\times) & D {\rm\, is \, split}\\
F^\times & D {\rm\, is\,not\, split}
\end{array}\right.
$$
where $\lambda$ is the similitude character.
\item
$$H^1_D(F)\simeq\left\{
\begin{array}{ll} O(V_F)(F) & D {\rm\, is \, split}\\
E^1 & D {\rm\, is\,not\, split}
\end{array}.\right.
$$
In particular, the non-identity connected component of
the disconnected algebraic group $H^1_D$
does not have $F$-points for the non-split $D$.
\end{enumerate}
\end{Prop}

It is easy to describe the  irreducible representations of $H_D(F)$.
By abuse of notations we shall write $H_D$ for $H_D(F)$
when there is  no ambiguity.

\begin{Prop}\label{rep:of:HD:1}
 Let $\eta$ be a character of $H^c_D\simeq E^\times$.
\begin{enumerate}
\item
Any irreducible representation of $H_D$ is a direct summand of some
two-dimensional representation $\Ind^{H_D}_{H_D^c} \eta$.
\item
If $\eta\neq \eta^\sigma$ then   $\tau_\eta^+=\Ind^{H_D}_{H_D^c} \eta$ is irreducible.
\item
If $\eta=\eta^\sigma$ then $\Ind^{H_D}_{H_D^c} \eta$
is a sum of two one-dimensional representations.
We denote them by $\tau_\eta^\pm$ .
\item
For any two characters $\eta$ and $\eta'$
$$\dim\Hom_{H_D^c}(\tau^\pm_\eta,\eta')=
\left\{\begin{array}{ll}
1& \eta'\in\{\eta,\eta^\sigma\}\\
0 & {\rm otherwise}
\end{array}.
\right.
$$
\end{enumerate}
\end{Prop}

\subsubsection{Labeling}\label{labeling}

We wish to distinguish between the two representations  $\tau_\eta^\pm$
when $\eta=\eta^\sigma$.
If $D$ splits we choose labeling such that
$$\tau_\eta^+|_{H_D^1(F)}=1, \quad \tau_\eta^-|_{H_D^1(F)}=\sgn.$$

If $D$ does not split then labeling is equivalent to
a choice of  $\eta_F$ such that $\eta=\eta_F\circ \Nm$. There
exist two such characters: $\eta^\pm_F$ where
$\eta^-_F=\chi_{E/F} \eta^+_F$. For any choice of $\eta^+_F$ define
$$\tau_\eta^\pm=\eta^\pm_F\circ \lambda.$$
Clearly, replacing $\eta^+_F$ by $\eta^-_F$ changes the labeling.

\subsection{The group $H_F$ and its representations.}

Define $H_F$ to be a group of
similitude automorphisms of $(V_F,s_F)$.
One has $\lambda(H_F)=\Nm_{E/F}(E^\times).$

If $D$ splits then $H_F\simeq H_D$. In particular
any irreducible representation of $H_F$ is a constituent
of $\Ind^{H_F}_{H^c_F} \eta$.

If $E$ is a split quadratic algebra over $F$
and $\sigma\in Aut_F(E)$ is a non trivial
automorphism we fix an isomorphism
$$H^c_F\simeq \G_m\times \G_m, \quad \sigma(x,y)=(x^{-1}y,y).$$
Any character of $\eta$ of $H_F^c$ can be written
in the form $\mu\boxtimes \eta_F$, where $\mu$ and $\eta_F$ are characters of
$\G_m$.

\begin{Lem}\label{rep:of:HD:split}
\begin{enumerate}
\item
$(\mu\boxtimes \eta_F)^\sigma= \mu^{-1}\boxtimes \mu \eta_F$.
In particular, $\eta$ is Galois invariant if and only if
$\mu=1$.
\item
$\tau^\pm_{\mu\boxtimes\eta_F}=(\eta_F\circ\lambda)\otimes  \tau^\pm_{\mu\boxtimes 1}$.
We shall write $\tau^\pm_{\mu}$ for $\tau^\pm_{\mu\boxtimes 1}$.

\item  $\tau^\pm_\mu=(\mu\circ \lambda) \otimes \tau^\pm_{\mu^{-1}}$.
\end{enumerate}
\end{Lem}


\subsection{The groups $H_{D_K}, H_K$ and $H^0_{D_K}, H^0_K$. }
Define $H_{D_K}$ and $H_K$ to be the group of similitude automorphisms of
$(V_{D_K},s_{D_K})$ and of $(V_K,s_K)$ respectively.

The similitude characters are denoted by
$\lambda_{D_K}$ and $\lambda_K$.  Also denote
$$H_{D_K}^0(F)=\{ h\in H_{D_K}(K): \lambda_{D_K}(h)\in F^\times \}.$$
$$H_K^0(F)=\{ h\in H_{K}(K): \lambda_K(h)\in F^\times \}.$$

Note the following properties:
\begin{enumerate}

\item
There are  natural embeddings
 $$H_D\hookrightarrow H_{D_K}^0, \quad H_F\hookrightarrow H_{K}^0.$$

\item
$$\lambda_K(H_K(F))=\Nm_{L/K} L^\times, \quad
\lambda_K(H^0_K(F))=\Nm_{L/K} L^\times\cap F^\times =F^\times.$$

\item
If $D_K$ splits over $K$ then
$$H_K\simeq H_{D_K}, \quad H^0_K\simeq H_{D_K}^0.$$
\end{enumerate}

\section{ Theta correspondences for similitude dual pairs}

\subsection{Howe duality} Let us recall the notion of
 the abstract Howe duality after Roberts, [R].
Let  $A$ and $B$ be reductive groups over a local field of
characteristic zero and let $\rho$ be a  smooth representation of $A\times B$.
For any smooth irreducible representation
$\tau$ of $A$ (resp. $B$) the maximal $\tau$ isotypic
 component of $\rho$ has the form $\tau\boxtimes \Theta(\tau)$
for some smooth representation $\Theta(\tau)$ of $B$ (resp. of $A$),
possibly zero.

We say that Howe duality holds for the triplet $(\rho,A,B)$ if for any irreducible
representation $\tau$ of $A$ (resp. $B$) such that $\Theta(\tau)\neq 0$
the maximal semisimple quotient $\theta(\tau)$ of $\Theta(\tau)$ is
irreducible. The representations $\Theta(\tau)$ and $\theta(\tau)$ are
called big theta lift and small theta lift of $\tau$ respectively. Theta correspondence is another name for Howe duality.

Let $(H^1,G^1)$ be a classical dual pair of isometry groups
acting on the spaces $V$ and $W$ respectively
and let $(H,G)$ be the corresponding similitude groups with the similitude
character $\lambda$. Define
$$H^+=\{h\in H: \lambda(h)\in \lambda(G)\}, \quad
G^+=\{g\in G: \lambda(g)\in \lambda(H)\}$$
and
$$R=\{(h,g)\in H\times G: \lambda(h)=\lambda(g)\}\subseteq H^+\times G^+.$$

Let $\omega_\psi$ be a Weil representation of $\widetilde{Sp}(V\otimes W)$
associated to a non-trivial additive character $\psi$.
Any splitting $s:H^1\times G^1\rightarrow \widetilde{Sp}(V\otimes W)$
defines the representation $\omega_{\psi,s}=\omega_\psi\circ s$ of $H^1\times G^1$.

The Weil representations $\omega_{\psi,s}$ has been suitably extended
from $H^1\times G^1$ to $R$   for symplectic-orthogonal
pairs in [R] and   for the quaternionic
unitary dual pairs where one of the spaces is hyperbolic
even-dimensional in [GT].

\begin{Prop}([R, GT])  \label{isometry:similitude:plus}
\begin{enumerate}
\item
Howe duality for the triplet  $(\omega_{\psi,s}, G^1,H^1)$
implies Howe duality for the triplet
$$(\Omega^+=\ind^{G^+\times H^+}_{R} \omega_{\psi,s}, G^+,H^+).$$
\item
If $\tau$ is an irreducible representation of $H^+$ such that
$\tau|_{H^1}=\oplus k\tau_i$,
then $\Theta(\tau)|_{G^1}=\oplus_i k\cdot \Theta(\tau_i)$
and if $\Theta(\tau_i)=\theta(\tau_i)$
for some $i$ then $\Theta(\tau)=\theta(\tau)$.
\end{enumerate}
\end{Prop}

\begin{Remark}
Proving his results, Roberts has made the assumption that
the restriction of irreducible representations from
$G^+$ to $G^1$ is multiplicity free.
This assumption holds for symplectic-orthogonal dual pairs ([AP]),
 but  does not necessarily hold for quaternion dual pairs.
The proof in [GT], Prop. $3.3$ that works for all classical dual pairs
does not require this assumption.

Although [GT] deals with smooth representations
over non-archimedean fields, the proof of  Prop. $3.3.$
applies in the case $F$ is archimedean as well.  Indeed,
it uses only Frobenius reciprocity and the technique
of restriction a representation to a subgroup of finite index.
\end{Remark}

Another general result was proven for isometry theta correspondence
in [MVW] and generalized for a similitude theta correspondence in [GT].

\begin{Thm}\label{theta:of:supercuspidal}
Let $F$ be a non-archimedean field.
For any supercuspidal representation $\tau$ of $H^+$,
its theta lift $\Theta(\tau)$ is either zero or irreducible.
In particular $\Theta(\tau)=\theta(\tau)$.
\end{Thm}

\subsection{Orthogonal-symplectic  dual pairs}

Howe duality does not hold for a general similitude orthogonal-symplectic  triplet $(\Omega=\ind^{G\times H}_R \omega_{\psi,s},H,G)$.
However it does hold in the case of interest in this paper.
We have the following proposition

\begin{Prop} \label{isometry:dimV:less:dimW}
Let $(H^1=SO(V,q), G^1=Sp(W))$ be an
orthogonal-symplectic dual pair
where
$(V,q)$ be an even-dimensional quadratic space
and  $\dim V\le \dim W$.
Then  Howe duality for the triplet $(\omega_{\psi,q},H^1,G^1)$
implies  Howe duality for the triplet $(\Omega=\ind^{G\times H}_{G^+\times H} \Omega^+,H,G)$.
\end{Prop}

\begin{proof}
Note that $H^+=H$.
If $\disc(V,q)$ is a split algebra then $G^+=G$ and there is nothing to prove.
Assume now that $\disc(V,q)=E$ is a field. In this case $G^+$ is a subgroup
of $G$ of index $2$.

It has been shown in [R], Prop. $1.2.$ that the following statements are
 equivalent:
\begin{enumerate}
\item  Howe duality for
$(\Omega,H,G)$ holds.
\item for any irreducible representation $\pi$ of $G^1$ and
any $g\in G\backslash G^+$  at most one of the representations
$\pi, \pi^g$ has a non-zero theta lift to $H^1$.
\end{enumerate}

Let us show that the statement $(2)$ above holds.
Otherwise there exists an irreducible representation
$\pi$ of $G^1$ and an element $g\in G$ with
$\lambda(g)=a\notin \Nm_{E/F}(E^\times)$ such that
$$\theta_{\psi,q}(\pi)\neq 0, \quad \theta_{\psi,q}(\pi^g)\neq 0$$
Hence  $\theta_{\psi,aq}(\pi^g)$ is non-zero
representation of $SO(V, aq)$.
The spaces $(V,aq)$ and $(V,q)$ belong to the different
Witt towers. It follows from the conservation
principal  recently proven in [SZ] for all local fields that
$$\dim (V,q)+\dim (V, aq)\ge 2\dim W+2.$$
This contradicts the condition $\dim(V)\le \dim (W)$.
\end{proof}

Let us get back to the notations of Sections $2-4$.
The group $(H^1_F,G^1_F)$ is an orthogonal-symplectic dual pair
inside $Sp(V_F\otimes_F W_F)$.
Given a polarization $W_F=X_F\oplus Y_F$
there is a   canonical splitting served by [Ku]
$$i_{s_F}: H^1_F\times G^1_F \hookrightarrow \widetilde{Sp}(V_F\otimes_F W_F).$$
The pullback of the Weil representation $\omega_\psi$ of
$\widetilde{Sp}(V_F\otimes W_F)$
defines a representation $\omega_{s_F,\psi}$
of $H^1_F\times G^1_F$.

We shall realize the Weil representation $\omega_{s_F,\psi}$ of
$H^1_F\times G^1_F$ on the space of the Schwarz functions
$S(V_F\otimes X_F)$.
The group $H_F\times M^1_F$
acts naturally on the space $S(V_F\otimes X_F)$. The action of $H_F\times P_F$ is given by  the usual formulas:

$$\left\{
\begin{array}{ll}
\omega_{s_F,\psi}(h) \phi)(x)=\phi(h^{-1} x) & h \in H^1_F\\
\omega_{s_F,\psi}(u(S))\phi(x)=
\psi(\langle x, u(S)x\rangle)\phi(x) &  S\in\Hom(X_F,Y_F), S=-S^\ast\\
\omega_{s_F,\psi}(m(a))\phi )(x)=
\chi_{E/F}(\det(a)) |\det(a)|\phi(x a) & a \in GL(X_F)
\end{array}
.\right.
$$
The representation $\omega_{s_F,\psi_F}$ can be extended to the group
 $$R_F=\{(h,g)\in H_F\times G_F: \lambda(h)=\lambda(g)\}$$
by  defining
$$\omega_{s_F,\psi}(h,g) \phi(x)=|\lambda(h)|^
{-\frac{ \dim V_F\cdot  \dim W_F}{8}}
\omega_{s_F, \psi}(1,g')\phi(h^{-1} x)=|\lambda(h)|^{-1}
\omega_{s_F, \psi}(1,g')\phi(h^{-1} x),
$$
where
$$g'=
\left(\begin{array}{ll}
\lambda(h) &0\\ 0 & 1
\end{array}\right)g\in G^1_F.$$
\begin{Remark} \label{compare:to:GanTakeda}
The definition above agrees with the definition of Weil representation
in [GTak]. Indeed, the similitude factor of a matrix $g\in GSp_4$
in [GTak] is defined to be $\det(g)^{1/2}$, while
here $\lambda(g)=\det(g)^{-1/2}$. The
groups $R_D$ and the group $R$ in [GTak] are the
same subgroups of $GO(V_F)\times GSp(W_F)$
and the representations are identical.
\end{Remark}

Let $\psi_K=\psi\circ tr_{K/F}$ be a non-trivial additive character of $K$.
Similarly to above $H^1_K\times G^1_K$ is a dual pair in $\widetilde{Sp}(V_K\otimes W_K)$

We denote by  $\omega_{s_K,\psi_K}$ the realization of the Weil representation of
$H^1_K\times  G^1_K$ on the space $S(V_K\otimes Y_K)$ such that
$$\left\{
\begin{array}{ll}
\omega_{s_K,\psi_K}(h,1) \phi)(x)=\phi(h^{-1} x) & h \in H^1_K\\
\omega_{s_K,\psi_K}(1, u(S))\phi )(y)=
\psi_K(\langle x, u(S) x\rangle_K)\phi(x) & S\in \Hom(X_K,Y_K), S=-S^\ast\\
\omega_{s_K,\psi_K}(1, m(a))\phi )(y)=
\chi_{L/K}(a) |a|_K\phi(xa) & a \in GL(X_K)
\end{array}.
\right.
$$
\vskip 10pt

Similar to the previous cases, we extend $\omega_{s_K,\psi_K}$  to
the group
$$R_K=\{(h,g)\in H_K\times G_K: \lambda_K(h)=\lambda_K(g)\}$$
by
$$\omega_{s_K,\psi_K}(h,g)=|\lambda(h)|_K^
{-\frac{ \dim_K V_K\cdot  \dim_K W_K}{8}}=
\omega_{s_K, \psi_K}(1,g')\phi(h^{-1} x)=|\lambda(h)|_K^{-1/2}
\omega_{s_F, \psi}(1,g')\phi(h^{-1} x).
$$
where
$$g'=
\left(\begin{array}{ll}
\lambda(h) &0\\ 0 & 1
\end{array}\right)g\in G^1_K.$$

\vskip 10pt

The representation  $\omega^0_{\psi_K,s_K}$ will denote the  restriction
of $\omega_{\psi_K,s_K}$ to $R^0_K=R_K\cap (H_K\times G^0_K)$.

\subsection{The induced representations $\Omega_E, \,\Omega_L$}
The compactly induced   representations
$$\ind^{H_F\times G_F}_{R_F}\omega_{s_F,\psi},
\quad \ind^{H_K\times G_K}_{R_K}\omega_{s_K,\psi_K}
$$
no longer depend  on the character $\psi$ nor on the Hasse
invariants of $s_F,s_K$  but only on the
discriminant of the form and hence will be denoted by
$\Omega_E$ and $\Omega_L$  respectively.
Indeed, for any $a\in F^\times$  we denote by $\psi_a$ the additive character
defined by $\psi_a(x)=\psi(ax)$. Let $g_a$ be any element of $G_F$
whose similitude factor is $a$. From the formulas above it follows easily that
 $$\omega_{s_F,\psi_a}\simeq \omega_{as_F,\psi}\simeq \omega_{s_F,\psi}^{g_a}.$$
Hence the induction of $\omega_{\psi,s_F}$
to $H_F\times G_F$ does not depend on $a$.
Similar argument applies for  $\omega_{s_K,\psi_K}.$

We also define  $\Omega^0_L=\ind^{H^0_K\times G^0_K}_{R^0_K}\omega^0_{s_K,\psi_K}$.

\begin{Prop}\label{Howe:classical}
Howe duality holds for the triplets $(\Omega_E, H_F,G_F), (\Omega_L,H_K,G_K)$ and
$(\Omega_L^0,H^0_K,G^0_K).$
\end{Prop}
\begin{proof}
Howe duality for the isometry triplet $(\omega_{s_F,\psi}, H^1_F,G^1_F)$ is proved in [Ya2].
Since $\dim V_F\le \dim W_F$ Howe duality
for the triplet $(\Omega_E,H_F,G_F)$
follows from the Proposition \ref{isometry:dimV:less:dimW}.

Howe duality for the triplet $(\omega_{s_K,\psi_K},H_K^1,G_K^1)$
is well known. The convenient reference is [Ca] for non-archimedean
field. When $F$ is archimedean, Howe duality holds by the general result
of Howe in [H].

 Since $\dim V_K=\dim W_K$, Howe duality
for the triplet $(\Omega_E,H_F,G_F)$
follows from Proposition \ref{isometry:dimV:less:dimW}.

Finally, $\lambda(H_K^0(F))=\Nm_{L/K}(L^\times)\cap F^\times =F^\times$
and $\lambda(G_K^0(F)=F^\times$. Hence
$(H^0_{K})^+=H^0_K, (G^0_K)^+=G^0_K$ and
Howe duality for the triplet $(\Omega_L^0,H^0_K,G^0_K)$ also holds.
\end{proof}

\subsection{Quaternionic unitary dual pair}

The pair $H^1_D\times G^1_D$ constitute a commuting pair
inside $Sp(V_D\otimes W_D)$. This  is a dual pair if and only if $h(D)=1$.

Assume that $W_D$ admits a polarization $X_D\oplus Y_D.$
This is always the case when $F$ is non-archimedean or if $F=\C$.
If $F=\R$ the space $W_D$ must have  signature $(1,1)$.

The group $H_D\times GL(X_D)$ acts naturally on the space of
Schwarz functions $S(V_D\otimes X_D)$.

We shall denote by $\omega_{s_D,\psi}$ the
realization of the Weil representation
of $H^1_D\times G^1_D$ on the space  $S(V_D\otimes X_D)$ such that
$$
\left\{
\begin{array}{ll}
\omega_{s_D,\psi}(h)\phi(x)=\phi(h^{-1}y)  & h\in H^1_D\\
\omega_{s_D,\psi}(m(a))\phi(x)=\chi_{E/F}(\Nm_{D/F}(a))|\Nm_{D/F}(a)|
\phi(xa), & a\in GL(X_D)\\
\omega_{ s_D,\psi}(u(S))\phi(x)=
\psi(\langle x, u(S)x \rangle)\phi(x) & S\in \Hom(X_D,Y_D): S=-\ol{S}^\ast
\end{array}
.\right.
$$

This definition agrees with the one given in [Ya2]. We extend $\omega_{s_D,\psi}$ to
 $$R_D=\{(h,g)\in H_D\times G_D: \lambda(h)=\lambda(g)\}$$
by setting
$$\omega_{s_D,\psi}(h,g) \phi(x)=
|\lambda(h)|^{-\frac{\dim_F D\cdot \dim_D V_D\cdot \dim_D W_D}{8}}
(\omega_{s_D, \psi}(1,g')\phi)(h^{-1}x)=
|\lambda(h)|^{-1}
(\omega_{s_D, \psi}(1,g')\phi)(h^{-1}x)
$$
where
$$g'=
\left(\begin{array}{ll}
\lambda(h) &0\\ 0 & 1
\end{array}\right)g\in G^1_D.$$

Note that the center of $R_D$ acts trivially. The representation
$\ind^{H_D\times G_D}_{R_D}\omega_{\psi,s_D}$ does not depend on $\psi$ and $s_D$
but only on its discriminant $E$ and hence will be denoted by $\Omega^D_E$.

The lemma below follows easily from the explicit Weil representation formulas and from Lemma
\ref{compatibility:polarizations}.
\begin{Lem}\label{weil:F=weil:D:split}
If $D$ splits then for compatible polarizations of $W_D$ and $W_F$
the canonical isomorphism $S(V_D\otimes X_D)\rightarrow S(V_F\otimes X_F)$
defines an isomorphism of representations
$\omega_{s_D,\psi}\simeq \omega_{s_F,\psi,}.$
\end{Lem}

\begin{Cor} If $D$ splits then
$$H_D\simeq H_F,\quad G_D\simeq G_F, \quad \Omega^D_E\simeq \Omega_E$$
and Howe duality for the triplet $(\Omega^D_E,H_D,G_D)$ is equivalent to Howe duality for the triplet $(\Omega_E,H_F,G_F)$
\end{Cor}

\begin{Prop}
Howe duality holds for $(\Omega^D_E,H_D,G_D)$.
\end{Prop}

\begin{proof}
In view of the last corollary it remains to consider  the case where $h(D)=-1$.
In this case $\lambda(H^1_D(F))=F^\times$ and hence $G^+_D=G_D.$
The Howe duality for  $(\Omega^D_E,H_D,G_D)$ follows now from
Howe duality for $(\omega_{\psi,s_D},H_D^1,G^1_D)$ that is proved in [Ya2].
\end{proof}

\section{Explicit Theta correspondence}
For any representation $\tau$ of $H_D$ we denote
by $\Theta^D_E(\tau)$ and $\theta^D_E(\tau)$ the representations
of $G_D$ that are  the big and
small theta lifts of $\tau$ respectively.
Similarly,
for an irreducible  representation $\tau$ of $G_K$ (resp. $G^0_K$) we
denote by  $\Theta_L(\tau)$ and $\theta_L(\tau)$ ( resp.
$\Theta^0_L(\tau)$ and $\theta^0_L(\tau)$) the representations
of $H_K$ (resp. $H^0_K$) which are the big and small theta lifts of $\tau$.

In this section we shall give more details on the
theta correspondence $\theta^D_E$, $\theta_L$ and
$\theta^0_L$ when $F$ is non-archimedean local field.

\subsection{Explicit theta correspondence $\theta^D_E$}
The theta lift $\theta^D_E$ will be used to define
the Arthur packet on $G_D$. Hence it is desirable
to know this theta lift as explicitly as possible.

To study the restriction problem we will apply see-saw duality
technique which makes use of big theta lift $\Theta^D_E(\tau)$ rather than
small theta lift $\theta^D_E$. Thus it is important to determine $\Theta^D_E(\tau)$ as well.
First we prove the following reduction lemma

\begin{Lem}\label{reduction}
\begin{enumerate}
\item
Let $D$ be a non-split algebra and $\eta=\eta^\pm_F\circ Nm_{E/F}$.
Then
$$\Theta^D_E(\tau^\pm_\eta)=
(\eta^\pm_F\circ \lambda^{-1}) \otimes \Theta^D_E(1).$$
\item
Let $E=F\times F$ and $\eta=\mu\boxtimes \eta_F$.
Then
$$\Theta^D_E(\tau^\pm_\eta)=
(\eta_F\circ \lambda^{-1}) \otimes \Theta^D_E(\tau^\pm_\mu)$$
\end{enumerate}
and the same relations hold for $\theta^D_E$.
\end{Lem}

\begin{proof}
Both statements follow immediately from the fact that
for any character $\chi$ of $F^\times$
$$\tau \boxtimes \Theta(\tau)\simeq
\left((\chi\circ \lambda) \otimes \tau\right)
 \boxtimes
\left((\chi\circ \lambda^{-1}) \otimes \Theta^D_E(\tau)\right)$$
as $R_D$ modules.
\end{proof}

\begin{Thm} \label{thetaHDGD:nonsplit}
Let $E$ be a field extension of $F$.
\begin{enumerate}
\item For any irreducible representation $\tau$ of $H_D$,
the representation $\Theta^D_E(\tau)$ is a non-zero irreducible
representation of $G_D$. In particular
$\Theta^D_E(\tau)=\theta^D_E(\tau)$.
\item Let $D$ be a non-split algebra.
\begin{enumerate}
\item
If $\eta\neq \eta^\sigma$ then
$\Theta^D_E(\tau_\eta^+)$ is supercuspidal representation.
\item
If $\eta=\eta^\sigma=\eta^\pm_F\circ Nm$ then
$$\Theta^D_E(\tau_\eta^\pm)=(\eta^\pm_F\circ \lambda^{-1}) \otimes
J_{P_D} (\chi_{E/F}|\cdot|^{1/2}\circ Nm_{D/F}).$$
\end{enumerate}
\item
Let $D$ be a split algebra.
\begin{enumerate}
\item If $\eta\neq \eta^\sigma$ then
$\Theta^D_E(\tau_\eta^+)=J_{Q}(\chi_{E/F}|\cdot|, \pi(\eta)|\cdot|^{-1/2}),$
where $\pi(\eta)$ is the dihedral supercuspidal representation
of $GL_2$ associated to the character $\eta$.

\item If $\eta=\eta^\sigma$ then
$\Theta^D_E(\tau_\eta^+)$ is the unique irreducible quotient of
$I_{B}(\chi_{E/F}|\cdot|, \chi_{E/F}, \eta_F|\cdot|^{-1/2})$
and $\Theta^D_E(\tau_\eta^-)$ is supercuspidal.
\end{enumerate}
\end{enumerate}
Here the representation $\pi(\eta)$ denotes the dihedral supercuspidal
representation of $GL_2$ with respect to a character $\eta$.
\end{Thm}

\begin{proof} If $E$ is a field then the group $H_D$ is compact modulo its
center.  Hence, the irreducibility of
$\Theta^D_E(\tau)$  follows from Theorem \ref{theta:of:supercuspidal}.
Yasuda in [Ya2] has shown that $\theta_D(\tau)$ is not zero
for any representation $\tau$ of $H^1_D$.
 Therefore, the non-vanishing follows from Prop.
 \ref{isometry:similitude:plus}, Part $(2)$.

Let us prove the second part. If
$\eta\neq\eta^\sigma$ then the restriction
of $\tau^+_\eta$ to $H^1_D$ is a sum of two non-trivial
representations. Yasuda has shown that the lift
of a non-trivial representations to $G^1_D$ is non-trivial and supercuspidal.
Hence, by Prop.  \ref{isometry:similitude:plus}, Part $(2)$
 the representation  $\theta^D_E(\tau_\eta^+)$ is also supercuspidal.

If $\eta=\eta^\sigma=\eta_F^\pm\circ \Nm_{D/F}$ then, by the reduction lemma above,
it is enough to determine the theta lift of the trivial  representation.
We construct a map

$$T\in \Hom_{H_D\times M_D}\left(
(\Omega^D_E)_{U_D},
 1 \boxtimes
((\chi_{E/F}|\cdot|)\circ \Nm_{D/F}) \boxtimes |\cdot|^{1/2}\right)$$
by
$$T(f)=\int\limits_{F^\times} f(1, m(1,t))(0) |t|^{-1/2}  \, d^\times t.$$
The equivariance properties are easily checked.

Hence by Frobenius reciprocity  the theta lift
$\theta^D_E(1)$
is a subrepresentation  of
$$\Ind^{G_D}_{P_D} \delta_{P_D}^{-1/2}
(\chi_{E/F}|\cdot|\circ \Nm_{D/F})\boxtimes  |\cdot|^{1/2}$$
and hence it is a Langlands quotient of
$\Ind^{G_D}_{P_D} (\chi_{E/F}|\cdot|^{1/2})\circ \Nm_{D/F}\boxtimes 1,$
i.e., $$J_{P_D}(\chi_{E/F}|\cdot|^{1/2}\circ \Nm_{D/F}\boxtimes 1).$$

The third  part is proved in [GI] Prop. $A.8$. The small
discrepancy in the notations is resolved by remark
\ref{compare:to:GanTakeda}.
\end{proof}
Next we consider the case where $E$ and hence $D$ are split algebras.
Recall that any representation of $H_D$ has the form $\tau^\pm_{\eta}$, where $\eta=\mu\boxtimes \eta_F$.

\begin{Thm}\label{thetaHDGD:split}
Let both $D$ and  $E$ be split algebras.
\begin{enumerate}
\item
If $\mu\neq 1,|\cdot|^{\pm 2}$
then
$$\Theta^D_E(\tau_\eta^+)=\theta^D_E(\tau_\eta^+)=\bigl(\eta_F \circ \lambda^{-1} \bigr)\otimes I_{Q}(\mu).$$
\item
If $\mu=|\cdot|^{\pm 2}$ then
$$\theta^D_E(\tau_\mu^+)=\bigl(\eta_F \cdot |\cdot|^{\pm 1}\bigr)\circ \lambda^{-1}.$$
\item
If $\mu=1$ then  $\Theta^D_E(\tau_\mu^\pm)$ are both irreducible and
$$\Theta^D_E(\tau^+_\eta)=
(\eta_F\circ \lambda^{-1})\otimes J_{Q}\bigl(|\cdot|\boxtimes (|\det|^{-1/2} \otimes \Ind^{GL_2}_B 1) \bigr),
\quad
\Theta^D_E(\tau^-_\eta)= (\eta_F\circ \lambda^{-1})\otimes
J_{P}\bigl((|\det|^{1/2} \otimes St)\boxtimes |\cdot|^{-1/2}\bigr).$$
\end{enumerate}
\end{Thm}
Lemma \ref{reduction} reduces the proof of this theorem to the $\eta_F=1$ case.
\begin{proof}
The key step in the proof is the following lemma.
\begin{Lem}
For  $\mu\neq |\cdot|^{2}$ there is an injective  map of $G_D$ modules
$$\Hom_{H_D}(\Omega^D_E,\Ind^{H_D}_{H_D^0}\mu\boxtimes 1) \hookrightarrow I_{Q_D}(\mu)^*.$$
\end{Lem}

\begin{proof}
By Frobenius reciprocity
$$\Hom_{H_D}(\Omega^D_E,\Ind^{H_D}_{H_D^0}\mu\boxtimes 1)=
\Hom_{H^0_D}(\Omega^D_E,\mu\boxtimes 1).$$

The restriction of $\Omega^D_E$ to $H^0_D\times G_D$
is in fact the Jacquet module $\Omega^D_E$ with respect to
the maximal parabolic subgroup $H^c_D$ in $H_D$.
The result of Gan and Takeda in [GTak] implies the following
filtration  of $H_D^c\times G_D$ modules
$$0\rightarrow J \rightarrow \Omega_E^D\rightarrow \Omega_E^D/J\rightarrow 0.$$

Here
$J=\ind^{H^c_D\times G_D}_{H^c_D\times Q_D} S(F^\times \times F^\times)$
where the action of $H^c_D\times Q_D$ on $ S(F^\times \times F^\times)$
is given by
$$(h(a,r), l(t,g)\phi)(x,y)=|r\det(g)^{-1}| \phi(xa\det(g)b^{-1},yr\det(g)). $$
The quotient $\Omega/J$ is isomorphic to $S(F^\times)$ and the action of $H^c_D\times G_D$ is
given by
$$(h(a,r),g)\phi(x)=|r|^{-1}|a|^2 \phi(xr\det(g)).$$

In particular, if $\mu\neq |\cdot|^2$ one has $\Hom_{H^c_D}(\Omega/J,\mu\boxtimes 1)=0$
and hence by Lemma 9.4 of [GG1] there is an isomorphism of $G_D$ modules
$$\Hom_{H^c_D}(\Omega^D_E,\mu\boxtimes 1)=\Hom_{H^c_D}(J,\mu\boxtimes 1)=(\Ind^{G_D}_{P_D} V)^\ast$$
where $V^\ast=\Hom_{H^c_D}(S(F^\times\times F^\times),\mu\boxtimes 1).$

A straightforward computation shows that $V$ is a one-dimensional space
on which $Q$ acts by $\mu^{-1}\boxtimes  \mu \circ \det$.
\end{proof}

To derive the theorem from the last lemma  assume  first that $\mu \neq 1$ so that
$\tau_\mu=\Ind^{H_D}_{H^c_D} \mu \boxtimes 1$ is irreducible.
 By the lemma, there exists a surjective map
$I_{Q}(\mu)\twoheadrightarrow \Theta^D_E(\tau_{\mu\boxtimes 1})$.

By Lemma \ref{structure:of:induced} we conclude that $\Theta^D_E(\tau_{\mu\boxtimes 1})=\theta^D_E(\tau_{\mu\boxtimes 1})=I_{Q}(\mu)$
for $\mu\neq |\cdot|^{\pm2}$.
It also follows that
$\theta^D_E(\tau_{|\cdot|^{-2} \boxtimes 1})=|\cdot|\circ \lambda$.
By Lemma \ref{rep:of:HD:split}, $\tau_{|\cdot|^{2} \boxtimes 1} \simeq \bigl(| \cdot | ^2 \circ \lambda \bigr)\otimes \tau_{|\cdot|^{-2} \boxtimes 1}$. Lemma \ref{reduction} implies now that
$$\theta^D_E(\tau_{|\cdot|^2 \boxtimes 1})=
(|\cdot|^2\circ \lambda^{-1})\otimes \theta^D_E(\tau_{|\cdot|^{-2}})=
|\cdot|\circ \lambda^{-1}.$$

Assume now that $\mu=1$ so that
$\Ind^{H_D}_{H^c_D} \mu \otimes 1=\tau^+_1\oplus \tau^-_1$.
By the lemma above there is a surjective map
$I_{Q_D}(1)\rightarrow \Theta^D_E(\tau^+_1)\oplus \Theta^D_E(\tau_1^-)$
and hence using Lemma \ref{structure:of:induced} both
$$\{ \Theta^D_E(\tau^+_1), \Theta^D_E(\tau_1^-)\}=
\{J_{Q} \bigl(|\cdot|\boxtimes | (\det|^{-1/2} \otimes \Ind^{GL_2}_{B_2} 1)\bigr),
J_{P} \bigl((|\det|^{1/2} \otimes St)\boxtimes |\cdot|^{-1/2} \bigr) \}$$
are irreducible.

On the other hand, the representation $\tau^+_1$ restricted to
$H^1_D$ is  the trivial representation whose lift $\Theta_{s_D,\psi}(1)$
to $G^1_D$ is determined by [Ya1]. It equals
$J^{G^1_F}_{P^1}(\Ind^{GL_2}_{B_2} 1,1/4)$.   Hence by (\ref{isometry:similitude:plus}), Part $(2)$
$$\Theta^D_E(\tau^+_1)=J_{Q} \bigl(|\cdot|\boxtimes | (\det|^{-1/2} \otimes \Ind^{GL_2}_{B_2} 1)\bigr),
\quad
\Theta^D_E(\tau^-_1)= J_{P} \bigl((|\det|^{1/2} \otimes St)\boxtimes |\cdot|^{-1/2} \bigr).$$
\end{proof}

Note that in the only case where $\theta^D_E(\tau)\neq \Theta^D_E(\tau)$
the representation $\tau$ is not unitary.
In fact it is  easy to prove that
 $\Theta^D_E(\tau)$ is not one dimensional. For $\mu=|\cdot|^{\pm2}$ it follows from the proof above that
$\Theta^D_E(\tau^+_{\mu \boxtimes 1})= I_{Q}(\mu)$.

Finally we can write explicitly the theta lift
of unramified representations.

\begin{Prop}\label{thetaHDGD:unramified}
 Let $D$ be a split algebra and let  $\eta$ be an  unramified character. Then,
$\theta^D_E(\tau_\eta^+)$ is unramified representation and
 \begin{enumerate}
\item If $E$ is a field and $\eta=\eta_F\circ Nm_{E/F}$
then $\theta^D_E(\tau_\eta^+)$  is the unique irreducible quotient of
$$I_{B}(\chi_{E/F}|\cdot|,\chi_{E/F}; \eta_F |\cdot|^{-1/2}).$$
\item If $E=F\times F$,  and $\eta=\mu\boxtimes \eta_F$
then $\theta^D_E(\tau_\eta^+)$ is the unique irreducible quotient of
$$I_{B}(|\cdot|,\mu; \eta_F |\cdot|^{-1/2}).$$
\end{enumerate}
\end{Prop}

Part one is proved in  [GI], A.8, while
the second part follows from Theorem \ref{thetaHDGD:split}.

\subsection{Twisted Jacquet modules of $G_D$}
For any irreducible representation $\Pi$ of $G_D$ denote its wave front by
$$\hat F(\Pi)=\{E\subset D: \exists \,T: \disc(T)=E \,\rm{and}\,
\Pi_{U_D,\Psi_T}\neq 0\}.$$

\begin{Prop} \label{localWF}
Let $\tau$ be an irreducible representation of $H_D$.
\begin{enumerate}
\item
$\hat F(\theta^D_E(\tau))=\{E\}.$

\item If $\disc(T)=E$ then
$\theta^D_E(\tau)_{U_D,\Psi_T}$ equals $\tau^\vee\otimes \chi_{D,-1}$
where $\Ker \chi_{D,-1}=H^c_D$ whenever
either $D$ is split  and $-1\notin \Nm_{E/F}(E^\times)$ or
$D$ is non-split and $-1\in \Nm_{E/F}(E^\times)$. Otherwise
$\chi_{D,-1}$ is a trivial character of $H_D$.
\end{enumerate}
\end{Prop}

\begin{proof}

Fix non-zero vectors $e\in X_D, e^\ast\in Y_D$ such that
$h_D(e,\ol{e^\ast})=1$. Note that
$V_D\otimes_D X_D\simeq V_D\otimes e.$

In the course of the proof define for any
$t\in \G_m$ by $\tilde t=tId_{X_D}+Id_{Y_D}\in G_D$.

There is an isomorphism of $H_D\times P_D$ modules
$$B:\Omega^D_E\rightarrow S((V_D\otimes e)\oplus F^\times)$$
given by
$$B(\phi)(v\otimes e,y)=\phi(1,\tilde y)(v\otimes e).$$
By the formulas of the Weil representations
$$\left(\omega_{s_D,\psi}(u(S))-\Psi_{T}(u(S))\right)
 B(\phi)(v\otimes e,y)=$$
$$\psi(y \tr_{D/F} \sigma(s_D(\ol{v},v))h_D(e,\ol{S(e)})-\tr_{D/F} (TS))
B(\phi)(v\otimes e,y)=$$
$$\psi(\tr_{D/F}
(y \sigma(s_D(\ol{v},v))-
\sigma(T(e^\ast, \ol{e^\ast})) h_D(e,\ol{S(e)}))B(\phi)(v\otimes e,y).$$
The expression above vanishes for all the skew-Hermitian
forms $S$ if and only if  $B(\phi)$ is supported on the set
$$A_{s_D, T}=\{ (v\otimes e, y):
y s_D(\ol{v}, v)=T(e^\ast,\ol{e}^\ast)\}.$$
Hence, the restriction of functions from
$S(V_D\otimes X_D)$ to $A_{s_D,T}$ defines an isomorphism
of  $H_D\times GU(Y_D,T)$ modules
$$(\Omega^D_E)_{U_D,\Psi_{T}}\simeq S(A_{s_D,T}).$$

If $\disc(s_D)\neq \disc(T)$ then $A_{s_D,T}=\emptyset$ and hence
${(\Omega^D_E)}_{U_D, \Psi_T}=0$.
If  $\disc(s_D)=\disc(T)$, the skew-Hermitian left $D$-modules
$(Y_D,T)$ and $(\ol{V_D} , s_D)$ are equivalent.
Hence, there exists an element $v_0\in V_D$ such that
$s_D(\ol{v_0},v_0)=T(e^\ast, \ol{e^\ast})$.

There is a natural bijection  of the sets  $GU(\ol{V_D},s_D)\simeq H_D$ and
$A_{s_D,T}$  via
$$h\rightarrow (h v_0, \lambda(h^{-1})).$$

Using this isomorphism we identify $(\Omega^D_E)_{U_D,\Psi_{T}}$ with $S(H_D)$.
By the formulas of the Weil representation, the action of
$H_D\times GU(Y_D,T)\subset H_D\times M_D$  on $S(H_D)$ is

$$\omega_{s_D,\psi}(h_1,h_2)\phi(a)=\chi_{E/F}\circ\Nm_{D/F}(h_2)
 |\lambda(h_1^{-1})|\cdot|\lambda(h_2)|
\phi(h_1^{-1}ah_2).$$
The character $\chi_{E/F}\circ \Nm_{D/F}$ on $H_D$  equals to $\chi_{D,-1}.$
Hence, for any $\tau$ of $H_D$ the isotypic component of $\tau$ in
$(\Omega^D_E)_{U_D,\Psi_T}$ is
$\tau\boxtimes (\tau^\vee\otimes \chi_{D,-1}).$
The proposition now follows.
\end{proof}

\subsection{Explicit theta correspondence $\theta_L$}
This correspondence is well-known.
Consider a  character $\eta_L:L^\times \rightarrow \C$.
If $L$ is a split algebra over $K$  fix an isomorphism
$L^\times\simeq  K^\times \times K^\times$ such that
$\sigma(x,y)=(x^{-1}y,y)$.
Then, the character $\eta_L$ has the form $\mu_K\boxtimes \eta_K$ so that
 $$\eta_L(x,y) =\mu_K(x)\eta_K(y).$$
Obviously, $\eta_L$ is Galois invariant if and only if  $\mu_K=1$.

\begin{Prop} Let $\pi$ be an irreducible  representation of $G_K$.
Then $\Theta_L(\pi)=\theta_L(\pi)$. More precisely,
\begin{enumerate}
\item
Let $L$ be a field. Then $\Theta_L(\pi)=0$ unless $\pi=\pi(\eta_L)$
for some $\eta_L:L^\times \rightarrow \C^\times$
and $\Theta_L(\pi(\eta_L))=\tau_{\eta_L}^+.$
\item Let $L=K\times K$.
Then $$\theta_L(\pi)=\left\{
\begin{array}{ll}
 \tau_{\mu_K\boxtimes \eta_K}^+ & \pi=
\Ind^{G_K}_{P_K} \mu_K\boxtimes \eta_K, \mu_K \neq |\cdot|^{\pm 1}\\
\tau_{|\cdot|_K\boxtimes \eta_K }^+ & \pi=\eta_K\circ \lambda^{-1}_K\\
0 & {\rm otherwise}
\end{array}
.\right.
$$
\end{enumerate}
\end{Prop}

\subsection{Relation between $\Theta_L$ and $\Theta^0_L$}
The relation between $\Theta_L$ and $\Theta^0_L$  is given in the
following proposition whose proof is identical to the proof of Prop.
\ref{isometry:similitude:plus}, part $(2)$.

\begin{Prop}\label{theta:isometry:restriction:L}
\begin{enumerate}
\item Let $\pi$ be an irreducible
 representation of $G^0_K$. Then $\Theta^0_L(\pi)$
is an irreducible representation  of $H^0_K$.
\item
Let $\pi$ be an irreducible  representation of $G_K$ such that
$\pi|_{G^0_K}=\oplus \pi_i$, sum of irreducible representations. Then
$\Theta_L(\pi)|_{H^0_K}=\oplus \Theta^0_L(\pi_i)$.
\end{enumerate}
\end{Prop}

\section{Local parameters and local packets}\label{local:parameters}

In this section we  describe the structure of the non-tempered
Arthur packets on $SO(V)$
and define them using the theta correspondence described above.

\subsection{Local parameters and local packets on $G_D$}\label{sec:definition:packets}
Let $F$ be a local field and let $W'_F$ be the Weil-Deligne group of $F$.
\begin{Def} The  local  Arthur parameter of $\theta_{10}$ type
is a map $$\Psi:W'_F\times SL_2(\C)\rightarrow Sp_4(\C),$$
where the image of $W'_F$ is bounded and
the image of a unipotent element of $SL_2(\C)$ is conjugated to a
short root unipotent element of $Sp_4(\C)$.
\end{Def}
The centralizer of the image of  $SL_2(\C)$ in $Sp_4(\C)$ is
the group $O_2(\C)$. Hence, by restriction, the parameter $\Psi$  gives
rise to a Langlands parameter
$\Phi :W_F\rightarrow O_2(\C)$. Obviously $\Phi$ determines $\Psi$.

The parameter
$\Phi:W_F\rightarrow O_2(\C)$ determines a quadratic algebra $E$ over $F$ as follows.
If the image of $\Phi$ is contained in $SO_2(\C)$ then
$E=F\times F$. Otherwise there exists a quadratic field extension
$E$ such that $\Phi(W_E)\subset SO_2(\C).$ Denote by $\sigma$
the non-trivial automorphism in $\Aut(E/F)$.
By class field theory $\Phi$ determines the character
$\eta:E^\times\rightarrow \C^\times$.
To stress this dependence we shall write  $\Psi_{E,\eta}, \Phi_{E,\eta}$
for $\Psi$ and $\Phi$ as above.
Note that the parameters $\Psi_{E,\eta}$ and $\Psi_{E,\eta^\sigma}$ are conjugate
and the same is true for $\Phi_{E,\eta},\Phi_{E,\eta^\sigma}$.

The dual  Langlands group of  $Z_D\backslash G_D$,
where $D$ runs over quaternionic algebras,
is $Sp_4(\C)$. Hence, by Arthur's conjecture the
parameter $\Psi_{E,\eta}$ gives rise
to a set $A^D_{E,\eta}$ of unitary admissible representations of $G_D$.
Similarly, the parameter $\Phi_{E,\eta}$ gives rise
to a set $L^D_{E,\eta}$ of unitary admissible representations of $H_D$.
 The unions
$$L_{E,\eta}=\cup_{D\supset E} L^D_{E,\eta},\quad
A_{E,\eta}=\cup_{D\supset E}A^D_{E,\eta} $$
must stay  in bijection with $\widehat S_{E,\eta}$, the set of characters of $S_{E,\eta}$, where
the  local component group $S_{E,\eta}$  of both
$\Psi_{E,\eta}$  and $\Phi_{E,\eta}$ is given by
$$S_{E,\eta}=
\left\{
\begin{array}{ll}\Z/2\Z\times \Z/2\Z & \eta=\eta^\sigma\\
\Z/2\Z & {\rm otherwise}
\end{array}.
\right.
$$

\vskip 5pt

The  L-packet $L^D_{E,\eta}$ of representations of $H_D$
is defined to consist of the
constituents of $\Ind^{H_D}_{H_D^c} \eta$.  More precisely,
$$
L^D_{E,\eta}=\left\{\begin{array}{ll}
\{ \tau_\eta^+,\tau_\eta^-\} &\eta=\eta^\sigma\\
\{\tau_\eta^+\} & \eta\neq \eta^\sigma
\end{array}
.\right.
$$
The authenticity of this construction will be
evident from  global considerations.

The structure of the packets $L^D_{E,\eta}$ and $A^D_{E,\eta}$ is identical.
Thus, it is natural to construct the Arthur packet $A^D_{E,\eta}$ of representations of $G_D$
using the theta-correspondence.
$$A^D_{E,\eta}=\{\Pi^\pm_\eta=\theta^D_E(\tau_\eta^\pm)\}.$$
Note that the  labeling of $L^D_{E,\eta}$ and hence of $A^D_{E,\eta}$
is not canonical when $D$ does not split and $\eta=\eta^\sigma$.

Let us describe the bijection $r$ between the sets
 $A_{E,\eta}$ (resp. $L_{E,\eta}$) and  $\widehat S_{E,\eta}$.

\begin{itemize}
\item
If $S_{E,\eta}=\Z/2\Z$
then
$$r(\Pi^+_\eta)=r(\tau^+_\eta)=
\left\{
\begin{array}{ll}
1 & D \, {\rm is \, split}\\
sgn &   D \, {\rm is \,not\,  split}
\end{array}
.\right.
$$
\item
If $S_{E,\eta}=\Z/2\Z\times \Z/2\Z$ then
$$r(\Pi^+_\eta)=r(\tau^+_\eta)=
\left\{
\begin{array}{ll}
1\otimes 1  & D \, {\rm is \, split}\\
1\otimes sgn &   D \, {\rm is \,not\,  split}
\end{array}
,\right.
\quad
r(\Pi^-_\eta)=r(\tau^-_\eta)=
\left\{
\begin{array}{ll}
sgn\otimes 1  & D \, {\rm is \, split}\\
sgn\otimes sgn &   D \, {\rm is \,not\,  split}
\end{array}
.\right.$$
\end{itemize}


\subsection{The local packet of $SO(U)\simeq G^0_K/F^\times$}
The packet of representations of $G^0_K(F)$  is defined as the set of
 constituents of a single irreducible representation $\pi$ of
$G_K(K)/F^\times$ after restriction to $G^0_K(F)/F^\times$.
We denote the corresponding parameter by $\Psi_{\pi}$.

\section{see-saw duality and the restriction theorem}\label{local:restriction}
In this section assume that the  algebra $D_K$ splits.

\subsection{See-saw duality}
There is a natural embedding
$i_G:G^0_K\simeq G^0_{D_K}\hookrightarrow G_D$.
The following proposition is straightforward.
\begin{Prop}\label{see:saw:preparation}
 Let $R^0_{D,K}=\{(h,g)\in H_D\times G_K^0: \lambda(h)=\lambda_K(g)\}$.
The group $R^0_{D,K}$ can be identified with the subgroup of $R^0_K$ and of $R_D$
via the imbeddings
$$ i_H\times id: R^0_{D,K}\hookrightarrow R^0_K,
\quad id\times i_G:  R^0_{D,K}\hookrightarrow R_D.$$
For compatible polarizations, the natural isomorphism
$S(V_K\otimes X_K)\rightarrow S(V_D\otimes X_D)$ defines an isomorphism
$$\omega^0_{s_K,\psi_K}|_{R^0_{D,K}}\simeq \omega_{s_D,\psi}|_{R^0_{D,K}}.$$
\end{Prop}
Using this we prove the see-saw duality theorem.
\begin{Thm}\label{sew:saw:theorem}
Let $\tau$ and $\pi$ be two irreducible representations of $H_D$ and
$G_{K}$ respectively. Then
$$\Hom_{G^0_K}(\Theta_E^D(\tau),\pi)=\Hom_{H_D}(\Theta^0_L(\pi),\tau).$$
\end{Thm}

\begin{proof}
Assume that after the restriction to $G^0_K$
the representation $\pi$ decomposes as
$\pi=\oplus \pi_i$. By definition one has
$$\Hom_{H_D}(\Theta^0_L(\pi_i),\tau)=
\Hom_{H_D}(\tau^\vee,\Hom_{G^0_K}(\Omega^0_L,\pi_i))=\Hom_{H_D\times G^0_K}(\Omega^0_L,\tau\boxtimes \pi_i).$$

Now, $\Omega^0_L|_{H_D\times G^0_K}=\ind^{H_D\times G^0_K}_{R^0_{D,K}}\omega_{s_K,\psi_K}$
and $\Omega^D_E|_{H_D\times G^0_K}=\ind^{H_D\times G^0_K}_{R^0_{D,K}}\omega_{s_D,\psi}.$
Hence, by Frobenius reciprocity and by Proposition \ref{see:saw:preparation}
the above equals

$$\Hom_{R^0_{D,K}}(\omega_{s_K,\psi_K}, \tau\boxtimes \pi_i)=
\Hom_{R^0_{D,K}}(\omega_{s_D,\psi},\tau\boxtimes \pi_i)=\Hom_{H_D\times G^0_K}(\Omega^D_E,\tau\boxtimes \pi_i)=
\Hom_{G^0_K}(\Theta^D_E(\tau),\pi_i).$$

Using Prop. \ref{theta:isometry:restriction:L} we obtain

$$\Hom_{H_D}(\Theta_L(\pi),\tau)=
\oplus_i \Hom_{H_D}(\Theta^0_L(\pi_i),\tau)=\oplus_i \Hom_{G^0_K}(\Theta^D_E(\tau),\pi_i)=
\Hom_{G^0_K}(\Theta^D_E(\tau),\pi).
$$
\end{proof}

\subsection{The main restriction theorem}

\begin{Def} We say that a character $\eta_L$ of $L^\times$
matches a character $\eta$ of $E^\times$ if there exists
$s\in Aut(L/K)$ such that $s(\eta_L)|_{E^\times}=\eta$.
\end{Def}

\begin{Thm} Let
$\Pi^\pm_\eta$ and $\pi$ be irreducible representations of $G_D$ and $G_K$
respectively.
\begin{enumerate}
\item $\Hom_{G_K^0}(\Pi^\pm_\eta,\pi)=0$ unless
$\pi=\pi(\eta_L)$ for some $\eta_L:L^\times \rightarrow \C$
such that $\eta_L$ matches $\eta$.

\item Assume that the condition in the first part holds.

\begin{enumerate}
\item If $\eta\neq \eta^\sigma$ then
$\dim \Hom_{G^0_K}(\Pi^+_\eta,\pi)=1$.
\item
If $\eta=\eta^\sigma$ and $\eta_L\neq \eta_L^\sigma$
then
 $$\dim \Hom_{G^0_K}(\Pi^+_\eta,\pi)=1, \quad
\dim \Hom_{G^0_K}(\Pi^-_\eta,\pi)=1.$$
\item
Assume that $\eta=\eta^\sigma, \eta_L=\eta_L^\sigma$.
In particular $\eta=\eta^\pm_F \circ \Nm_{E/F}$ and
$\eta_L=\eta_K \circ \Nm_{L/K}.$ Then
$\eta_K|_{F^\times}=\eta^+_F$ or $\eta^-_F$.

If $D$ splits then
$$\dim \Hom_{G^0_K}(\Pi^+_\eta,\pi)=1, \quad \dim \Hom_{G^0_K}(\Pi^-_\eta,\pi)=0.$$
If $D$ does not split then the labeling is not canonical
and  depends on the choice
$\eta^+_F$. More precisely, if  $\eta_K|_{F^\times}=\eta_F^\pm$
then
$$\dim \Hom_{G^0_K}(\Pi^\pm_\eta,\pi)=1, \quad
\dim \Hom_{G^0_K}(\Pi^\mp_\eta,\pi)=0.$$

\end{enumerate}
\end{enumerate}
\end{Thm}

\begin{proof}

First assume that $\Theta^D_E(\tau)=\theta^D_E(\tau).$
Then
$$\Hom_{G^0_K}(\Pi^\pm_\eta, \pi)=
\Hom_{G^0_K}(\Theta^D_E(\tau_\eta^\pm),\pi)=
\Hom_{H_D}(\tau_{\eta_L}^+,\tau_\eta^\pm)$$
when $\pi=\pi(\eta_L)$ and zero otherwise.
It is easy to describe the latter space using Frobenius reciprocity.

Indeed, in case $(a)$,
if $\eta_L\neq \eta^\sigma_L, \eta\neq \eta^\sigma$ one has
$$\Hom_{H_D}(\tau_{\eta_L}^+,\tau^+_\eta)=
\Hom_{H_D^c}(\eta_L, \eta)\oplus \Hom_{H_D^c}(\eta_L, \eta^\sigma)$$
that is one-dimensional if $\eta_L$ matches $\eta$ and zero otherwise.

In case $(b)$,
$\eta_L\neq \eta^\sigma_L$ but $\eta=\eta^\sigma$. One has
$$\Hom_{H_D}(\tau_{\eta_L}^+,\tau^\pm_\eta)=
\Hom_{H_D^c}(\eta_L, \eta)$$
that is one-dimensional if $\eta_L$ matches $\eta$ and zero otherwise.

Finally assume that $\eta_L=\eta^\sigma_L$ and that $\eta=\eta^\sigma$.
To prove $(c)$ note that the condition $\eta_L|_{E^\times}=\eta$ implies
$\eta_K\circ \Nm_{L/K}|_{E^\times}=\eta^\pm_F\circ \Nm_{E/F}|_{E^\times}$
and hence the restriction of  $\eta_K$ to $\Nm_{E/F}(E^\times)$ coincides
with $\eta^\pm_F$. In particular $\eta_K=\eta^\pm_F$.
Note that there are two possible choices for $\eta_K$ but their
restrictions to $F^\times$ coincide.

If $D$ splits then
$\tau_{\eta_L}^+$ restricted to $H^1_K$ is trivial
and hence the restriction to the subgroup
$H^1_D\subset H^1_K$ is also trivial.
Thus, $\tau_{\eta_L}^+$ restricted to $H_D$ equals
$\tau_\eta^+$.

For non-split $D$ we know already that $\eta_K=\eta^\pm_F$.
 Assume that $\eta_K=\eta^+_F.$ Then
$\tau_{\eta_L}^+=\eta_K\circ\lambda_K$ whose restriction to
$H_D$ is $\eta^+_F\circ \lambda=\tau_\eta^+$.  Hence
$$\dim \Hom_{G^0_K}(\Pi^+_\eta,\pi)=1, \quad \dim \Hom_{G^0_K}(\Pi^-_\eta,\pi)=0.$$
The case $\eta_K=\eta^-_F$ follows by the same argument.
\vskip 10pt

Finally, if $\Theta^D_E(\tau)\neq \theta^D_E(\tau)$ then by Propositions
\ref{thetaHDGD:nonsplit} and \ref{thetaHDGD:split}
one has $E=F\times F$,
 $\tau=\tau_{|\cdot|^{\pm 2}\boxtimes \eta_F}$ and
$\theta^D_E(\tau)=(\eta_F|\cdot|^\mp) \circ \lambda$. In this case the restriction of  $\theta^D_E(\tau)$ to $G^0_K$ is
$(\eta_F|\cdot|_K^{\mp 1/2})\circ \lambda_K$. The latter representation
is $\pi(|\cdot|_K^{\mp 1/2}\boxtimes \eta_F)$ and the  character
$|\cdot|_K^{\mp 1/2}\boxtimes \eta_F$ of $K^\times \times K^\times$
matches the  character $|\cdot|^{\mp 1}\boxtimes \eta_F$ of $F^\times \times F^\times$.

\end{proof}

\begin{Cor} Let $\Psi_1=\Psi_{E,\eta}$ and $\Psi_2=\Psi_\pi$.
Then
\begin{equation} \label{sum}
\sum_{U'\subset V'}\sum_{\Pi\in A_{E,\eta}}
\dim\Hom_{SO(U')}(\Pi,\pi)
\end{equation}
vanishes unless $\pi=\pi(\eta_L)$
for a character $\eta_L$ matching $\eta$.

If $\pi=\pi(\eta_L)$ with $\eta_L$ matching $\eta$ then
(\ref{sum}) equals
$$
\left\{
\begin{array}{ll}
4 & \eta= \eta^\sigma, \eta_L\neq \eta_L^\sigma \\
2& \rm{otherwise}
\end{array}
\right..
$$

\end{Cor}

\section{Global Parameters and global packets}\label{global:parameters}

\subsection{Global parameters}
Let $F$ be a number field. Denote by $pl(F)$ the set of
completions of $F$.
Let  $\mathcal L_F$ be the conjectural Langlands group
associated with $F$.
A non-tempered global  Arthur parameter of $\theta_{10}$ type  is a map
$$\Psi:\mathcal L_F\times SL_2(\C)\rightarrow Sp_4(\C),$$
where the image of the unipotent element of $SL_2(\C)$ is conjugated to the
short root unipotent element of $Sp_4(\C)$ and the image of $\mathcal L_F$
is bounded.

As before, it determines by restriction a global tempered parameter
$\Phi :\mathcal L_F\rightarrow O_2(\C)$. If $\Phi(\mathcal L_F)$ is contained
in $SO_2(\C)$ then its centralizer is not finite modulo the center
and hence  the parameter is  not expected to
 contribute to the discrete spectrum. Thus we assume that this is not the case.
In particular, there exists a unique  quadratic field extension $E$ over $F$,
with  the non-trivial automorphism  $\sigma\in Gal(E/F)$,
 such that $\Phi(\mathcal L_E)\subset SO_2(\C).$
Moreover,  by class field theory $\Phi$ determines a
unitary automorphic  character
$$\eta=\otimes_{w\in pl(E)} \eta_w:E^\times\backslash \I_E  \rightarrow \C^\times$$
with a trivial restriction to $\I_F$.
We shall write $\Phi_{E, \eta}$ and $\Psi_{E,\eta}$
for $\Phi,\Psi$ respectively.

For any place $v$ of $F$ there is an inclusion
$W'_{F_v}=\mathcal L_{F_v}\hookrightarrow \mathcal L_F$.
Thus, by composition, the parameters $\Psi_{E,\eta},\Phi_{E,\eta}$
give rise to a family of local parameters
$\Psi_{E_v,\eta_v},\Phi_{E_v,\eta_v}$ where
$E_v=E\otimes_F F_v$ is a quadratic algebra over $F_v$ and
$$\eta_v=
\left\{
\begin{array}{ll}
\eta_w & v=w\\
\eta_{w_1}\boxtimes \eta_{w_2} & v=w_1w_2
\end{array}.\right.
$$

\begin{Def}
The automorphic characters $\eta_1$ and $\eta_2$
will be called {\rm globally equivalent} if there exists $s\in Gal(E/F)$
such that $s(\eta_1)=\eta_2$.
The characters will be called almost everywhere equivalent
if for almost all places $v$ of $F$ there exists  $s_v\in Aut(E_v/F_v)$
such that ${\eta_1}_v=s_v({\eta_2}_v)$.
\end{Def}

\begin{Remark}
The parameters $\Psi_{E,\eta_1}$ and $\Psi_{E,\eta_2}$ (resp.
 $\Phi_{E,\eta_1}$ and $\Phi_{E,\eta_2}$ ) associated with the globally equivalent
characters are conjugate.
\end{Remark}

In the next section the following lemma will be used

\begin{Lem}\label{local:global:1}
Let $\eta_1,\eta_2: E^\times \backslash \I_E\rightarrow \C^\times$
be two almost everywhere equivalent unitary characters.
Then $\eta_1,\eta_2$ are
globally equivalent.
\end{Lem}

\begin{proof} Let $S$ be a finite set of places outside of which
the local equivalence holds. Denote by $L^S(\cdot,s)$ the
Euler product of the local L-functions over the places outside of $S$
and by $L_S$ the finite product of the local L-functions over the places
in $S$.

Consider the L-function
$L(\eta_1\eta^{-1}_2,s)L(\eta_1\sigma(\eta_2)^{-1},s).$
By assumption it equals
$$\zeta_E^S(s)L^S(\eta_1\sigma(\eta_1)^{-1},s)
L_S(\eta_1\eta^{-1}_2,s)L_S(\eta_1\sigma(\eta_2)^{-1},s)$$
and hence has a pole at $s=1$. Thus, one of the terms
$L(\eta_1\eta_2^{-1},s), L(\eta_1\sigma(\eta_2)^{-1},s)$ has a pole at $s=1$.
In particular, either $\eta_1=\eta_2$ or $\eta_1=\sigma(\eta_2)$.
\end{proof}

\subsection{Global packets}

Let $D$ be a quaternion division algebra over $F$ containing  $E$
and let $\eta=\otimes_{v\in pl(F)}
 \eta_v$ be an automorphic character of $\I_E$.
 We shall define several subsets of $pl(F)$.

\begin{itemize}
\item
Let $S_E$ be the  set of places $v$ of $F$ such that $E_v=E\otimes_v F_v$
does not split over $F_v$.
\item
Let $S_D$ be the set of places $v$ of $F$  at which
$D_v$ does not split over $F_v$. This set is always finite and of even cardinality.
Clearly, $S_D\subset S_E$.
 \item
Let $S_\eta$ be the set of places $v$ of $F$ such that
$\eta_v=\eta_v^\sigma$ for the non-trivial $\sigma$ in $Aut(E_v/F_v)$.
In particular, for any $D$ the local packets $A^{D_v}_{E_v,\eta_v}$ and
$L^{D_v}_{E_v,\eta_v}$
constructed in Section \ref{sec:definition:packets}
contain two elements for $v\in S_\eta$
and are singletons otherwise. The set $S_\eta$ is always infinite since any $v\in S_E$ such that
$\eta_v$ is unramified belongs to it.
\item Let $S_{-1}$ be the set of places $v$ of $F$
such that $-1\notin \Nm_{E_v/F_v}(E_v^\times)$.
This set is always finite and of even cardinality.

\end{itemize}

For the
one-dimensional skew-Hermitian space  $(V_D,s_D)$ over $D$ of discriminant
$E$ and for the two-dimensional {\sl hyperbolic} Hermitian space $(W_D,h_D)$ over
$D$ we denote
$$H_D(\A)=GU(V_D,s_D)(\A),\quad G_D(\A)=GU(W_D,h_D)(\A).$$

The parameters $\Phi_{E,\eta}$ and $\Psi_{E,\eta}$ give rise to the
packets $L^D_{E,\eta}$ and $A^D_{E,\eta}$ of representations of
$H_D(\A)$ and $G_D(\A)$ as follows.

For a collection of representations
$\epsilon=(\epsilon_v\in \widehat {S_{E_v,\eta_v}})$,
where $\epsilon_v$ is the trivial representation
for almost all $v$, define the
representations $\Pi_\eta^\epsilon$ of $G_D(\A)$ and
$\tau_\eta^\epsilon$ of $H_D(\A)$ by
$$\Pi_\eta^\epsilon=\otimes \Pi_{\eta_v}^{\epsilon_v}, \quad
\tau_\eta^\epsilon=\otimes \Pi_{\eta_v}^{\epsilon_v}.$$

Furthermore, for any finite set $S\subseteq S_\eta$
consider the collection $\epsilon_S=(\epsilon_v\in \widehat{S_{E_v,\eta_v}})$
such that

$$\Pi^{\epsilon_v}_{\eta_v}=\left\{
\begin{array}{ll}
\Pi^+_{\eta_v} & v\notin S\\
\Pi^-_{\eta_v} & v\in S
\end{array}
.\right.$$

We shall denote by $\tau_\eta^S$ and  $\Pi_\eta^S$
the representations of $H_D(\A)$ and $G_D(\A)$ respectively,
corresponding to the collection $\epsilon_S$.

The global Arthur packets for $G_D$ and Langlands packets for $H_D$ are
 defined to be
$$A^D_{E,\eta}=\{ \Pi_\eta^S:S\subseteq S_\eta, |S|<\infty\},\quad
L^D_{E,\eta}=\{ \tau_\eta^S: S\subseteq S_\eta, |S|<\infty\}.$$


\subsection{Arthur multiplicity formula}

The global component groups of the parameters $\Psi_{E,\eta}$  $\Phi_{E,\eta}$
 are isomorphic and are equal to
$$S_{E,\eta}=\pi_0(Cent_{^L{G}}(\Image \Psi_{E,\eta}))=
\left\{
\begin{array}{ll}\Z/2\Z \times \Z/2\Z& \eta=\eta^\sigma\\
\Z/2\Z & \eta\neq \eta^\sigma
\end{array}
\right.
$$
and there is a natural inclusion $i:S_{E,\eta}\hookrightarrow \Pi_v S_{E_v,\eta_v}$.

According to Arthur's conjecture for a collection of representations
$\epsilon=(\epsilon_v\in \widehat {S_{E_v,\eta_v}})$
where $\epsilon_v$ is trivial for almost all $v$, the multiplicities
of the representations $\Pi_\eta^\epsilon,\tau_\eta^\epsilon$ in the discrete
spectrum of $G_D$ and $H_D$ respectively  are equal to

\begin{equation}
\label{multiplicity:formula:conj}
m(\epsilon)=\frac{\sum_{s\in S_{E,\eta}} \Pi_v\epsilon_v(i(s))}{|S_{E,\eta}|}.
\end{equation}

\begin{Lem}
For an automorphic character $\eta$ and a finite set $S\subset S_\eta$
one has
$m(\epsilon^S)=m(\eta,S)$ where
$$m(\eta,S)=\left\{
\begin{array}{ll} 0 & \eta=\eta^\sigma, {\rm |S|\,\, is \,\, odd} \\
1 & {\rm otherwise}
\end{array}
.\right.$$
\end{Lem}

\begin{proof}
If $S_{E,\eta}=\Z/2\Z$ then
$\epsilon_S(-1)=(-1)^{|S_D|}=1$ since the set $S_D$ has even cardinality.

If $S_{E,\eta}=\Z/2\Z\times \Z/2\Z$ then
$$\epsilon_S(1,-1)=(-1)^{|S_D|}, \quad
\epsilon_S(-1,1)=(-1)^{|S|},
\quad
\epsilon_S(-1,-1)=(-1)^{|S\triangle S_D|}.$$
Here $\triangle$ denotes the symmetric difference of the sets.

Hence, for any finite $S\subseteq S_\eta$ the RHS of
 (\ref{multiplicity:formula:conj})
equals
$$\left\{
\begin{array}{ll} 0 & \eta=\eta^\sigma, {\rm |S|\,\, is \,\, odd} \\
1 & {\rm otherwise}
\end{array}
.\right.$$
\end{proof}

Recall that for $v\in S_D \cap S_\eta$   the
labeling of the representations in $L^{D_v}_{E_v,\eta_v}$ (and hence
in $A^{D_v}_{E_v, \eta_v}$) was not canonical.
For the multiplicity formula to hold we should choose the labeling
in a coherent way.
If $\eta\neq \eta^\sigma$ then the labeling of  $L^{D_v}_{E_v,\eta_v}$ for
$v\in S_D\cap S_\eta$ can be chosen arbitrary.
However, if
$\eta=\eta^\sigma$ then $\eta=\eta_F\circ \Nm_{E/F}$
for an automorphic character $\eta_F$.
The character $\eta_F$ is defined up to multiplication by $\chi_{E/F}$.
The choice of $\eta_F$ fixes the labeling for every $v\in S_D$.
Replacing $\eta_F$ by
$\chi_{E/F}\eta_F$ will change the labeling simultaneously at all the places in $S_D$.

Our next task will be to prove that Arthur's multiplicity conjecture
holds for all the representations  in $L_{E,\eta}$ and $A_{E,\eta}$.

\section{Discrete spectrum of $H_D$}\label{discrete:spectrum:of:H}

In this section we shall describe explicitly the decomposition
of the space $L^2(Z_{D}(\A)H_D(F)\backslash H_D(\A))$.
Since $E$ is assumed to be a field, the space
$Z_{D}(\A)H_D(F)\backslash H_D(\A))$ is compact and hence the
space of square integrable functions decomposes discretely.

\begin{Thm}\label{discrete:spectrum:of:HD}
$$L^2(Z_{H_D}(\A)H_D(F)\backslash H_D(\A))=
\bigoplus_{\eta/\sim, \eta|_{\I_F}=1} V(\eta),\quad
V(\eta)= \bigoplus_{S\subseteq S_\eta, |S|<\infty}
m(\eta,S)\tau_{\eta}^S.$$
\end {Thm}

\begin{proof}
First note that all the representations in the sum are non-isomorphic.
Indeed $\tau_{\eta_1}^{S_1}\simeq \tau_{\eta_2}^{S_2}$ if and only if
for each place $v$ of $F$ there exists $\sigma_v\in Aut(E_v/F_v)$ such that
${\eta_1}_v={\eta_2}_v^{\sigma_v}$. By Lemma
\ref{local:global:1} the characters
$\eta_1,\eta_2$ are equivalent. Hence, also $S_1=S_2$.

The representations $\tau_\eta^S$ with $m(\eta,S)=1$
are realized in the discrete spectrum $L^2 (Z_{H_D}(\A)H_D(F)\backslash H_D(\A))$
 by an Eisenstein series.
$$
 E_\eta: \Ind^{H_D(\A)}_{H_D^c(\A)}\eta\rightarrow
 L^2(Z_{H_D}(\A)H_D(F)\backslash H_D(\A))$$
defined by
\begin{equation}\label{Eisen:HD}
E_\eta(\phi)(h)=\sum_{s \in H_D^c(F)\backslash H_D(F)} \phi(s h).
\end{equation}
It follows  immediately from the definition that
 $$\Ker (E_\eta)=\bigoplus_{(\eta,S):m(\eta,S)=0} \tau_\eta^S.$$
In particular, we have shown one inclusion. To establish the equality
take $f$ such that $(f,E_\eta(\phi))=0$ for all
$\eta, \phi\in \Ind^{H_D(\A)}_{H_D^c(\A)}\eta$.
One has
$$(f,E_\eta(\phi))_{H_D}=\int\limits_{H_D^c(\A)\backslash H_D(\A)}
\left(
\int\limits_{Z_{H_D}(\A)H_D^c(F)\backslash H_D^c(\A)}
 f(rh) \overline{\eta(r)}\, dr
\right)
 \overline{\phi(h)}\, dh.$$

If the integral vanishes for all $\phi$ then the inner integral vanishes
for almost all $h$  and for all $\eta$. Hence $f=0$.
\end{proof}

\begin {Remark}
In particular, for a fixed equivalence class of $\eta$,
the space $V(\eta)$ is a full nearly
equivalence class of cuspidal representations of $H_D(\A)$.
This shows that the L-packets $L_{E,\eta}$ were defined correctly.
\end{Remark}

Since the global packets $L_{E,\eta}$ and  $A_{E,\eta}$
have the same structure it is
natural to attempt to construct the automorphic representations in $A_{E,\eta}$
using the automorphic representations of $L_{E,\eta}$.
The global theta correspondence method provides such a construction.

\section{The global theta correspondence}
Whenever $m(\eta,S)=1$ we shall construct an automorphic realization of
$\Pi_\eta^S$ using the global theta correspondence.

Recall the construction of the global theta correspondence for the
dual pair $(G_D,H_D)$. Let $\psi:F\backslash \A\rightarrow \C^\times$
be a non-trivial additive character.
The representation
$$\left(\omega_{s_D,\psi}=
\otimes_v \omega_{s_{D_v},\psi_v}, R_D(\A), S(V_D\otimes X_D)(\A)\right)$$
admits an automorphic realization
$$\theta\in \Hom_{R_D(\A)}(S(V_D\otimes X_D)(\A),
\mathcal A(R_D(F)\backslash R_D(\A)))$$
given by
$$\theta(\phi)(r)=\sum_{x\in (V_D\otimes X_D)(F) } \omega_{s_D,\psi}(r)\phi(x).$$


\subsection{The global theta correspondence from $H_D$ to $G_D$}

For a cuspidal representation
$\tau$ of $H_D(\A)$ define a map
$$\theta^D_E:\omega_{\psi,s_D}\boxtimes \overline{\tau}
\longrightarrow \mathcal A(G_D(F)\backslash G_D(\A))$$
as follows:
For $g\in G_D(\A)$ such that  $\lambda(g)\in \lambda(H_D(\A))$ take $h\in H_D(\A)$
such that $(h,g)\in R_D(\A)$ and define
$$\theta^D_E(\phi\otimes  f)(g)=
\integral{H_D^1} \theta(\phi)(h_1h,g) \overline{f(h_1h)} dh_1.$$
For  $\lambda(g)\notin \lambda(H_D)$  define
$\theta^D_E(\phi\otimes  f)(g)$ to be zero.

We denote by $\theta^D_E(\tau)$ the representation spanned by the functions
$$\{\theta^D_E(\phi\otimes f)(g), \quad \phi\in \omega_{s_D,\psi}, f\in \tau \}.$$

\begin{Thm}
\begin{enumerate}
\item
For any irreducible representation $\tau_\eta^S$ of $H_D(\A)$,
the representation $\theta^D_E(\tau^S_\eta)$
is irreducible, non-zero and contained in the discrete spectrum of $G_D$.
Moreover,
$$\theta^D_E(\tau^S_\eta)\simeq
\otimes_{v\notin S} \theta^{D_v}_{E_v}(\tau_{\eta_v}^+)\otimes \otimes_{v\in S} \theta^{D_v}_{E_v}(\tau_{\eta_v}^-).$$
\item
Let $D$ be a split algebra. Then
$\theta^D_E(\tau_\eta^S)$ is cuspidal unless $S=\emptyset$.
\item
Let $D$ be a non-split algebra. Then
$\theta^D_E(\tau_\eta^S)$ is cuspidal unless $\eta=\eta^\sigma$
 and $S\subseteq S_D$.
\end{enumerate}
\end{Thm}

\begin{proof} The non-vanishing follows from
the non-vanishing of the global theta lift for isometries,
proven in [Ya2].
 The irreducibility of the global lift
follows from the irreducibility of the local lifts.
Parts $(2)$ and $(3)$ follow from Lemma 1.3 of [S]
and Lemma 3.1 of [Ya2].
If the conditions in Parts (2) and (3) do not hold then applying the
square-integrability criterion  it is shown in [Ya1] and [KRS]  that the theta lifts are contained in the
discrete spectrum.
\end{proof}

In particular, for any $\eta,S$ one has
$$m(\Pi_\eta^S)\ge m(\eta,S).$$
The other inequality will be proved in Section \ref{WFandMF}.

\section{Wave front and the multiplicity formula}\label{WFandMF}
\subsection{The Fourier coefficient}
Fix a non-trivial automorphic additive character $\psi:F\backslash \A\rightarrow \C^\times$.
We identify the unitary characters of $U_D(F)\backslash U_D(\A)$ with
the space of skew-Hermitian forms on $Y_D(F)$.
Any skew-Hermitian form $T$ on $Y_D(F)$ gives rise to a form
on $Y_D(\A)$. Denote
$$\Psi_T(u(S))=\psi(\tr_{D(\A)/\A}(TS)).$$

The $M_D(F)$ orbits  on the space of characters stay in bijection
with the equivalence  classes
of {\sl locally isometric}
skew-Hermitian forms on $Y_D(F)$
and hence are parameterized by the quadratic algebras $E$ inside $D$.
The stabilizer in $M_D(\A)$ of the character associated with the form $T$ is isomorphic
to $GU(Y_D,T)(\A)$.

For any automorphic form $\varphi$ on $G_D$ define
its  Fourier coefficient with respect to the skew-Hermitian form $T$ by
$$F_T(\varphi)=\integral{U_D}\varphi(u)\overline{\Psi_T(u)} \,du.$$

For any automorphic representation $\Pi$ of $G_D(\A)$
the Fourier coefficient defines a map
 $$F_T\in \Hom(\Pi_{U_D,\Psi_T}, \mathcal A(GU(Y_D,T)(F)\backslash GU(Y_D,T)(\A))$$
by $$F_T(\varphi)(h)=F_T(h \varphi).$$

In particular, for $\Pi\simeq \Pi^S_\eta$ one has  $F_T(\Pi)=0$
unless $\disc T=E$. For $\disc T=E$ there is an isomorphism
$$GU(Y_D,T)\simeq GU(\ol{V_D},s_D)\simeq H_D$$
 and by Proposition \ref{localWF} $F_T(\Pi)$ is a quotient of
 $(\tau_\eta^S)^\vee\otimes\chi_{D,-1}$,
where
$${(\chi_{D,-1})}_v=
\left\{\begin{array}{ll}
sgn & v\in S_D \triangle S_{-1}\\
1 & {\rm otherwise}
\end{array}.
\right.
$$

\subsection{The wave front}

For any irreducible automorphic representation $\Pi$ define
the wave front set by
$$\hat F(\Pi)=\{ E: \exists \,T: \disc(T)=E, \varphi\in \Pi:
 F_T(\varphi)\neq 0\}.$$
If $\Pi=\otimes \Pi_v$ and $E\in \hat F(\Pi)$
then obviously $E_v\in \hat F(\Pi_v)$. We shall make use of the following important theorem by Jian-Shu Li, [Li].
\begin{Thm}\label{Li}
Let $\Pi$ be a cuspidal irreducible representation of $G_D(\A)$. Then
$\hat F(\Pi)\neq \emptyset.$
\end{Thm}

\begin{Prop}\label{globalWF}
Let $\Pi$ be an automorphic representation of $G_D(\A)$ which is nearly equivalent to $\Pi_\eta^\emptyset$.
Then, $\hat F(\Pi)=\{E\}$.
\end{Prop}

\begin{proof}
First note that  $\hat F(\Pi)\subseteq \{E\}$. Indeed, assume that
$E'\in \hat F(\Pi)$. For almost all places $v$ one has
 $\hat F(\Pi_v)=\hat F(\Pi^+_\eta)=\{E_v\}$
by Proposition \ref{localWF}.
Hence, the characters $\chi_{E/F}$ and $\chi_{E'/F}$ are nearly
equivalent. By the strong multiplicity one theorem $E=E'$.

It remains to show that $\hat F(\Pi)\neq \emptyset.$
For a cuspidal $\Pi$ this follows from Theorem \ref{Li}. Assume now that $\Pi$ is not contained  in the cuspidal
spectrum of  $G_D(\A)$. Then, restricting $\Pi$
to $G^1_D(\A)$ and  restricting the functions in $\Pi$
to $G^1_D(F)\backslash G^1(\A)$ we obtain a (possibly
reducible) representation of $G^1_D(\A)$ that is not
contained in the cuspidal spectrum of $G^1_D(\A)$.
By [Ya1] every irreducible residual representation
that is nearly equivalent to a constituent of $\Pi_\eta^\emptyset|_{G^1_D(\A)}$
is isomorphic to $\theta^D_{s_D,\psi}(1)$ and hence by
Lemma 4.12 in [Ya1] has some non-trivial Fourier coefficient.
Thus, $\hat F(\Pi)\neq \emptyset$.

\end{proof}

From this we can deduce Arthur's multiplicity formula.

\begin{Prop}\label{MF:theorem}
 For any $\eta$ and $S$ the multiplicity $m(\Pi^S_\eta)$ of $\Pi^S_\eta$
in the discrete spectrum of $G_D(\A)$ equals $m(\eta,S)$.
\end{Prop}
\begin{proof}

Assume that $m(\eta,S)=0$. Let us show that $m(\Pi_\eta^S)=0$. Indeed, if
$\Pi_\eta^S$ can be embedded into the discrete spectrum of $G_D$
then $\hat F(\Pi_\eta^S)=\{E\}$. Hence, for a skew-Hermitian form $T$
with discriminant $E$, the space $F_T(\Pi_\eta^S)$ defines a non-zero irreducible
automorphic representation of $H_D(\A)$ isomorphic to
$(\tau_\eta^S)^\vee\otimes \chi_{D,-1}.$
Since by assumption
 $\eta=\eta^\sigma$,  the above representation is
$$\tau_{\eta^{-1}}^{S\triangle (S_D\triangle S_{-1})}.$$
In particular $S\triangle (S_D\triangle S_{-1})$ has even cardinality and
 hence $S$ has even cardinality so $m(\eta,S)=1$. This is a contradiction.

Assume now that $(\eta,S)$ is such that $m(\eta,S)=1$.
Let us show that the multiplicity of $\Pi_\eta^S$ in the discrete
spectrum is $1$. The multiplicity is at least one because
one realization is given by the theta correspondence
$\Pi_\eta^S=\theta^D_E(\tau_\eta^S)$.
If there are  two embeddings
$$J_1, J_2 \in\Hom_{G_D(\A)}(\Pi_\eta^S, L_{disc}^2(Z_D(\A)G_D(F)\backslash G_D(\A)),$$
we know that $\hat F(J_1(\Pi_\eta^S)), \hat F(J_2(\Pi_\eta^S))=\{E\}$.
The image of the maps
$$F_T\circ J_i
\in \Hom_{U_D(\A)}(\Pi_\eta^S, L^2(H_D(F)\backslash H_D(\A))$$
define an automorphic irreducible  representation of $H_D(\A)$
isomorphic to $(\tau_\eta^S)^\vee\otimes \chi_{D,-1}$.
By the multiplicity one property for the discrete spectrum of $H_D$
there exists a constant $c\in \C^\times$ such that $F_T (J_1-cJ_2)(\Pi_\eta^S)=0$.
Hence, the automorphic representation $(J_1-cJ_2)(\Pi_\eta^S)$ does not support
any non-degenerate coefficients along $U_D(\A)$ and therefore by Proposition \ref{globalWF} it is zero. In other words, $J_1$ and $J_2$ are proportional and hence
the multiplicity of $\Pi_\eta^S$ in the discrete spectrum is one.
\end{proof}

\section{L-function}
Our next goal is to show that the constructed Arthur packets contain
the full nearly equivalence class of cuspidal representations.
Our approach exploits a Rankin-Selberg integral representation
of an $L$-function of degree $5$ of a cuspidal representation of $G_D(\A)$.
This is a generalization of  the $L$-function studied in [PS-R].

\subsection{Notations}
Below $F$ is a number field and $\A$ is its ring of adeles. For any place $v$ of $F$,
$F_v$ denotes the $v$-adic completion of $F$. If $F_v$ is non-archimedean $\mathcal O_v$
denotes the ring of integers of $F_v$,
$\varpi_v$ denotes a uniformizer inside $\mathcal O_v$ and $q_v$ denotes the cardinality of the residue field .
For any finite set of places $S$ we denote  $\A_S=\Pi_{v\in S} F_v$.

The group $Sp_4(\C)$ which is the L-group of $Z_D\backslash G_D$,  admits a
$5$ dimensional irreducible complex representation $\rho$, given
by the accidental isomorphism $PSp_4(\C)\simeq SO_5(\C)$ discussed above.
hLet
$\chi:E^\times\backslash \I_E\rightarrow \C^\times$
 be an automorphic character
and let $\Pi=\otimes_v \Pi_v$ be an irreducible representation of $G_D(\A).$
There exists a finite set of places $\Omega$ which includes all the archimedean places
and which contains $S_D$ such that for $v\notin \Omega$ the representation $\Pi_v$ is
unramified with a Satake parameter $t_{\Pi_v}$  and the character $\chi_v$ is unramified too.
 Define
$$L^\Omega(\Pi,\chi, \rho,s)=\Pi_{v\notin \Omega}
\det(1-\chi_v(\varpi_v)\rho(t_{\Pi_v})q_v^{-s})^{-1}.$$

Let $\Pi$ be an irreducible representation of $G_D(\A)$ nearly equivalent
to $\Pi_\eta^\emptyset$ for some
automorphic character $\eta$ of $E^\times\backslash \I_E$.
Hence $\Pi$ and $\Pi_\eta^\emptyset$ share partial L-functions
$L^\Omega(\cdot,\chi_{E/F},s)$ for a set $\Omega$ large enough.

By  Proposition  \ref{thetaHDGD:unramified}
$$L^{\Omega}(\Pi,\chi_{E/F},\rho,s)=
\zeta^\Omega(s-1)\zeta^\Omega(s)L^\Omega(\chi_{E/F},s)^2\zeta^\Omega(s+1)$$
and hence it  has a simple pole at $s=2$.
We shall show that this property characterizes
the cuspidal representations in the
nearly equivalence class of $\Pi_\eta^\emptyset$.

\subsection{Eisenstein series} Let $K$ denote the maximal  compact subgroup
of $G_D(\A)$.
For any  $K$-finite standard section
$f(\cdot,s)$ in the unitary induced representation
$\Ind^{G^1_D}_{P^1_D} |\Nm_{D/F}|^s$, consider
the associated Eisenstein series for $s$, whose real part is sufficiently large
$$E(g,f,s)=\sum_{\gamma \in P^1_D(F)\backslash G^1_D(F)} f(\gamma g,s).$$
The Eisenstein series admits a meromorphic continuation for the entire compex plane.

\begin{Thm}
For any standard section $f(g,s)$,  the Eisenstein series $E(g,f,s)$
has at most a simple pole
at $s=3/2$. The pole is attained by the spherical  section and
the residue is the constant function.
\end{Thm}

\begin{proof} For $D$ split this is proved in Theorem $3.1$ in [S],
and for non-split $D$ this is proved in [Ya1].
\end{proof}

Let us define the normalized Eisenstein series by
$$E^\ast(g,f,s)=\zeta(2s-1)\zeta(s+1/2)E(g,f,s)$$

\subsection{Rankin-Selberg integral}

Let $\Pi$ be an irreducible cuspidal representation of $G_D(\A)$ such that
$E\in \hat F(\Pi)$. For any $\varphi\in \Pi$, $\phi\in S(V_D\otimes X_D)$
and a section $f(\cdot,s)$ consider the integral
\begin{equation}\label{RS:integral}
\mathcal Z(\varphi,\phi,f,s)=
\integral{G^1_D} \ol{\varphi(g)}\theta^D_E(\phi)(g)E^\ast(g,f,s-1/2) \,dg.
\end{equation}

The function $\varphi$ is rapidly decreasing on
$G^1_D(F)\backslash G^1_D(\A)$.
In particular the integral converges absolutely and hence defines a meromorphic function
on $\C$.

\begin{Thm}
Let $\Omega$ be a finite set of places which includes all the archimedean places
and the set  $S_D$ such that
outside of $\Omega$ the representation $\Pi_v$ and the field extension
$E_v$ are unramified.
Let $\varphi=\otimes_v\varphi_v$ and $f=\otimes f_v$ be factorizable  data,
that is spherical outside of the set $\Omega$.
Then, for $Re(s)$ sufficiently large
\begin{equation}\label{RS:Lfunction}
\mathcal Z(\varphi,\phi,f,s)=
L^\Omega(\Pi,\chi_{E/F},\rho,s) d_\Omega(\varphi,\phi,f, s),
\end{equation}
where
$$d_\Omega(\varphi,\phi,f)=
\int\limits_{U_D(\A_\Omega)\backslash G^1_D(\A_\Omega)}
\ol{F_{s_D}(g\varphi) }\omega_{s_D,\psi}(g)(\phi)(r_0) f^*(g,s) dg$$

Moreover, for every $s_0\in \C$
there exist $\varphi,f,\phi$ such that
$d_\Omega(\varphi,\phi,f,s)$ defines a  holomorphic
non-zero function in the neighborhood of $s_0$.
\end{Thm}

The proof of this theorem will occupy the rest of this section.
Let us list some immediate corollaries:
\vskip 5pt
The zeta integral $\mathcal Z(\varphi,\phi,f,s)$ has meromorphic
continuation to the whole complex plane
and hence the identity \ref{RS:Lfunction}
can be used to define  the meromorphic continuation of
the partial L-function $L^\Omega(\Pi,\chi_{E/F},\rho,s)$.

For an irreducible cuspidal representation $\Pi$ of $G_D(\A)$ define
the representation $\theta^D_E(\Pi)$ of $H_D(\A)$ spanned
by the functions
$$\theta^D_E(\phi,\varphi)(h)=\integral{G^1_D}
\theta^D_E(\phi)(h,g_1g)
\ol{\varphi(g_1g)} dg_1, \quad \lambda(g)=\lambda(h),
\varphi\in \Pi, \phi\in \omega_{\psi,s_D}.$$

\begin{Cor} \label{Lfunction:hence:thetanonzero}
Let $\Pi$ be an irreducible
cuspidal representation of  $G_D(\A)$ such that $E\in \hat F(\Pi)$
and the finite set $\Omega$ is as above.
If $L^\Omega(\pi,\chi_{E/F},\rho,s)$ has a pole at $s=2$ then
$\theta^D_E(\Pi)\neq 0$.
\end{Cor}

\begin {proof}
Let $\varphi,\phi,f$ be functions such
that $d_\Omega(\varphi,\phi,f,s)$ is holomorphic,
nonzero around $s=2$. Hence
$\mathcal Z(\varphi,\phi,f,s)$ has a pole at $s=2$ and
the leading term of Laurent expansion of $\mathcal Z(\varphi,\phi,f,s)$ at $s=2$ is
is $Res_{s=2} E^\ast(g,f,s)\theta^D_E(\phi,\varphi)(1)$.
\end{proof}

If $D$ splits, the theorem was proven by [PS-R].
The proof in the case where $D$ does not split is almost identical and is sketched in the following three
subsections.

\subsection{Unfolding}
\begin{Prop}
For $Re(s)$ sufficiently large it holds
\begin{equation}\label{unfolded:integral}
\mathcal Z(\varphi,\phi,f,s)=\int\limits_{U_D(\A)\backslash G^1_D(\A)}
\ol{F_{s_D}(g\varphi)} f(g,s-1/2) \omega_{\psi,s_D}(g)\phi(r_0) \, dg,
\end{equation}
where $r_0$ is a fixed non-zero vector of $V_D\otimes X_D$.
\end{Prop}

\begin{proof}
Substituting the definition of the
Eisenstein series for $Re(s)$ sufficiently large we obtain
$$\mathcal Z(\varphi,f,s)=
\int\limits_{P^1_D(F)\backslash G^1_D(\A)}
\ol{\varphi(g)}
\sum_{x\in (V_D\otimes X_D)(F)}
\omega_{s_D,\psi}(g)\phi(x) f(g,s-1/2) dg=$$
$$
\int\limits_{U_D(F)\backslash G^1_D(\A)}
\ol{\varphi(g)}\omega_{s_D,\psi}(g)\phi(r_0) f(g,s-1/2) dg=$$
$$
\int\limits_{U_D(\A)\backslash G^1_D(\A)}
\ol{F_{s_D}(g\varphi)}
\omega_{s_D,\psi}(g)\phi(r_0) f(g,s-1/2) \, dg.$$

Since  $D$ does not split there are only two orbits of the action of $M^1_D(F)$ on
$V_D\otimes X_D(F)$: the zero orbit and the open orbit. The element $r_0$
is the representative of an open orbit. The contribution from the zero orbit vanishes because
of cuspidality of $\Pi$.
\end{proof}

\begin{Remark}
Note that collapsing the sum and the integration above
will be justified if we show that
the integral on the right hand side absolutely converges. We shall show it
in the next subsection.
\end{Remark}

\subsection{Unramified computation}
For a general $\Pi$,
the space $\Hom_{U_D(F_v)}(\Pi_v,\C_{\Psi_{\ol{s_D}}})$ is not
one-dimensional. Hence,
the functional $F_{{s_D}}$ is not necessarily factorizable.
However, the integral (\ref{unfolded:integral})
is factorizable due to the following striking proposition.

\begin{Prop}
Let $v\notin \Omega$ and let $v^0_v$
be a $K_v$-fixed vector of $\pi_v$.
Let $f^\ast(g,s)$ be a spherical section normalized by
$f^\ast(e,s)=\zeta(2s-1)\zeta(s+1/2)$.
For any  functional $L\in \Hom_{U_D(F_v)}(\Pi_v,\C_{\Psi_{s_D}})$ one has
$$\int\limits_{U_D(F_v)\backslash G^1_D(F_v)}
\ol{L(g\, v_v^0)}\, \omega_{\psi,s_D}(g)(\phi)(r_0\otimes e) f^*_v(g,s-1/2) dg =
L(\Pi_v,\chi_{E_v/F_v},\rho,s) L(v^0_v).$$
\end{Prop}

When $v\notin S_D$ one has $G^1_D\simeq Sp(4)$.
Hence, the proposition is a special case of the main theorem in
[PS-R] for $n=2$.
As a corolllary  we obtain the decomposition
$$\mathcal Z(\varphi,\phi,f,s)=
L^\Omega(\Pi,\chi_{E/F},\rho, s) d_\Omega(\varphi,\phi, f,s),$$
where
$$d_\Omega(\varphi, \phi,f, s)=\int\limits_{U_D(\A_\Omega)\backslash G^1_D(\A_\Omega)}
\ol{F_{s_D}(g\varphi)} \omega_{\psi,s_D}(g)(\phi)(r_0) f^*(g,s-1/2) \, dg.$$

The integral on the right hand side converges for $Re(s)$ sufficiently large.
Indeed, by Iwasawa decomposition its convergence is equivalent to the convergence
of
$$J(\varphi,\phi, s)=\int\limits_{M^1_D(\A_\Omega)}
 \ol{F_{s_D}(m\varphi)}\phi(m r_0)|\Nm_{D/F}(m)|^s\chi_{E/F}(\Nm_{D/F}(m))\, dm.$$

Since $\varphi$ is of moderate growth there exists a constant $k$
such that $F_{s_D}(m \varphi)\le C(\varphi_S) \|m\|^k$. Hence
for $Re(s)$ sufficiently large the right hand side of (\ref{RS:Lfunction}) converges absolutely.
This justifies the formal operation in the process of the unfolding.

\subsection{The ramified factor}
It remains to show that for  any $s_0\in \C$ there exists
$$\varphi\in \Pi,\quad
\phi\in S(V_D\otimes X_D)(F_\Omega), \quad f\in
\Ind^{G^1_D(\A_\Omega)}_{P^1_D(\A_\Omega)} |\Nm_{D/F}|^{s-1/2}$$
such that
$$d_\Omega(\varphi,\phi,f)=
\int\limits_{U_D(\A_\Omega)\backslash G^1_D(\A_\Omega)}
\ol{F_{s_D}(g\varphi) }\omega_{s_D,\psi}(g)(\phi)(r_0\otimes e) f^*(g,s-1/2) dg$$
is holomorphic  and does not vanish in a neighborhood of $s_0$.

Define as above
$$J(\varphi,\phi, s)=\int\limits_{M^1_D(\A_\Omega)}
 \ol{F_{s_D}(m\varphi)}\phi(m r_0)|\Nm_{D/F}(m)|^s\chi_{E/F}(\Nm_{D/F}(m))\, dm.$$
Suppose $\varphi$ is such that $F_{s_D}(\varphi)\neq 0.$


It is possible to find a Schwarz function $\phi$ whose support is small enough
to ensure the holomorphicity and non-vanishing of
$J(\varphi,\phi, s)$ around $s=s_0$. Then,
$$d_\Omega(\varphi,f,\phi,s)=
\int\limits_{K_\Omega} J(k \varphi, k \phi, s) f(k) \, dk.$$

Choose now a standard section $f$ whose restriction to $K$ has a
support which is
small enough to ensure the non-vanishing of
$d_\Omega(\varphi,f,\phi,s)$ around $s=s_0$.

\section{The nearly equivalence classes}\label{NEC}
We shall use the results of the previous section
to show that the constructed set of representations contains,
together with every cuspidal representation,
its  full nearly equivalence class of cuspidal representations.

\begin{Thm}
Let $\Pi$ be a cuspidal irreducible  representation of $G_D(\A)$
that is nearly equivalent to $\Pi_\eta^\emptyset$.
Then, there exists $S$ such that $\Pi=\theta^D_E(\tau_\eta^S)$.
\end{Thm}

\begin{proof}
 We start with a cuspidal  representation
$\Pi$ nearly equivalent to $\Pi_\eta^\emptyset$.
There is a finite set $\Omega$ such that
$$L^\Omega(\Pi,\chi_{E/F},\rho,s)=L^\Omega(\Pi^\emptyset_\eta,\chi_{E/F},\rho,s)$$
and hence  has a pole at $s=2$.
Moreover, by Proposition \ref{globalWF} one has
$\hat F(\Pi)=\{E\}$. Hence, by Corollary \ref{Lfunction:hence:thetanonzero}
 $\theta^D_E(\Pi)$ is a non-zero
irreducible representation  of $H_D(\A)$ which is
nearly equivalent to $\tau_\eta^\emptyset$.
By Proposition \ref{discrete:spectrum:of:HD}
one has $\theta^D_E(\Pi)=\tau_\eta^S$ for some $S$ with
$m(\eta,S)=1$. Hence $\Pi\simeq \Pi_\eta^S$. From Proposition
\ref{MF:theorem} it now follows that $\Pi=\theta^D_E(\tau_\eta^S)$.
\end{proof}

\section {Global theta correspondence from $G_K$ to $H_K$}

\subsection{ The pairs $(G_K, H_K)$ and $(G^0_K, H^0_K)$}

 Denote by $\A_K$ the ring of adeles of the field $K$.
The automorphic realization of the global Weil representation
$$(\omega_{\psi,s_K}, R_K(\A_K), S(V_K\otimes X_K(\A_K))$$
is given by
$$\theta(\phi)(h,g)=\sum_{x\in {V_K\otimes X_K}(K)}\omega_{s_K,\psi_K}(h,g)\phi(x).$$

For a cuspidal representation $\pi$ of $G_K(\A)$ define its theta lift to be the space
spanned by the functions of the form
$$\theta_L(\phi\otimes \varphi)(h)=
\int\limits_{G^1_K(K)\backslash G^1_K(\A_K)}
\theta(\phi)(h,g_1g)\ol{\varphi(g_1g)}\, dg_1,$$
where $\lambda(g)=\lambda(h), \phi\in \omega_{s_K,\psi_K}, \varphi\in \pi$.

This is a well-known lift. Below we summarize its properties.

\begin{Prop}
Let $\pi$ be an irreducible cuspidal representation of $G_K(\A)$.
 The following statements are equivalent.
\begin{enumerate}
\item
$\theta_L(\pi)\neq 0.$
\item
$L(\pi, Ad\otimes \chi_{L/K},s)$ has a pole at $s=1$.
\item
$\theta_L(\pi)=\tau^\emptyset_{\eta_L}$ for some character $\eta_L$.
\item
 $\pi=\pi(\eta_L)$
is a dihedral cuspidal representation for some character
$$\eta_L: L^\times \backslash \I_L\rightarrow \C^\times,
\quad \eta_L\neq s(\eta_L),\, s \in Aut(L/K).$$

\end{enumerate}
\end{Prop}

\begin{Remark}\label{vanishing:E=K}
\begin{enumerate}
\item
If $E=K$ then
$V_K$ is a split quadratic space, $\chi_{L/K}=1$. The L-function
$L(\pi,Ad,s)$ is entire   and hence $\theta_{L}(\pi)=0$
for any cuspidal representation $\pi$ of $G_K(\A)$.

\item
The automorphic realization of the representations $\tau_{\eta_L}^S$ of
 $H_K(\A_K)$ is given as in \ref{Eisen:HD} by
$$E_{\eta_L}:
\Ind^{H_K}_{H_K^c}\eta_L\rightarrow L^2(Z_{H_{K}}(\A)H_K(K)\backslash H_K(\A_K))$$
$$E_{\eta_L}(\phi_K)(h)=\sum_{H_K^c(K)\backslash H_K(K)} \phi_K(\gamma h)=
\sum_{s\in Aut(L/K)} \phi_K(sh).$$

\end{enumerate}
\end{Remark}

Let $\pi$ be an irreducible  cuspidal  representation of
$G_K(\A_K)$.
Consider a space of automorphic functions
on $G^0_K(\A)$ obtained by the restriction of functions in the space of $\pi$.
This space decomposes as direct sum of nearly equivalent
representations of $G^0_K(\A)$, that constitute an automorphic  packet.
Moreover, any automorphic packet on $G^0_K(\A)$ arises in this way.

\section{global see-saw identity}

\begin{Prop} Let $\tau$ and $\pi$ be two unitary cuspidal representations
of $H_D(\A)$ and $G_K(\A_K)$ respectively.
Assume that  the central character of $\tau$ is trivial and the
central character of $\pi$ has trivial restriction to $\I_F$.
Then, there is an equality of the Petersson inner products
$$(\theta^D_E(\phi\otimes f_\tau),f_\pi)_{G^0_K}=
(\theta_L(\phi\otimes f_\pi),f_\tau)_{H_D}.$$
\end{Prop}

\begin{proof}
In the course of this proof denote $G^{0,+}_K=G^0_K \cap G^+_D$. In particular

$$G^{0,+}(F_v)=
\left\{
\begin{array}{ll}
\{ g\in G_K(K_v): \det(g)\in F_v\} & v\in S_D\\
\{g \in G_K(K_v): \det(g)\in \Nm_{E_v/F_v}(E_v)\} & v\notin S_D
\end{array}
\right.
$$

Let
$$\mathcal C=(\A_F^\times)^2 \Nm(E^\times)\backslash
(\Pi_{v\in S_D}F_v) \Nm_{E/F}(\A^{S_D}_E).$$
The similitude characters of the groups $G_K^{0,+}(\A)$ and $H_D(\A)$
 induce isomorphisms
$$Z^0_{K}(\A)G^1_K(\A)G^{0,+}_K(F)\backslash G_K^{0,+}(\A)\simeq \mathcal C,
\quad
Z_{D}(\A)H^1_D(\A)H_D(F)\backslash H_D(\A)\simeq \mathcal C.$$

Fix the splitting maps
$\mathcal C\rightarrow  G^{0,+}_K(\A)$ and $\mathcal C\rightarrow  H_D(\A)$
and denote them by $c\mapsto g_c$ and $c\mapsto h_c$ respectively.



To obtain the global see-saw duality we write for $f_\tau\in \tau, f_\pi\in \pi$
$$(\theta^D_E(\phi\otimes f_\tau),f_\pi)_{G^0_K}=
\int\limits_{Z_K(\A) G^0_K(F)\backslash G^0_K(\A)}
\theta^D_E(\phi\otimes f_\tau)(g)\overline{f_\pi(g)} dg=
\int\limits_{Z_K(\A) G^{0,+}_K(F)\backslash G^{0,+}_K(\A)}
\theta^D_E(\phi\otimes f_\tau)(g)\overline{f_\pi(g)} dg=$$
$$\int\limits_{\mathcal C}
\integral{G_K^1}
\integral{H_D^1}
\theta_L(\phi)(g_1g_c,h_1h_c)
\ol{f_\tau(h_1h_c)}dh_1
\overline{f_\pi(g_1g_c)}dg_1 dc=$$
$$\int\limits_{Z_D(\A) H_D(F)\backslash H_D(\A)}
\theta(\phi\otimes f_\pi)(h)
\overline{f_\tau(h)} dh=
(\theta_L(\phi,f_\pi),f_\tau)_{H_D}.$$
\end{proof}

\section{The main global theorem}\label{global:restriction}

Let $K$ be a quadratic algebra and $D$ be a quaternion algebra
such that $S_D\subset S_K$.

Define the period integral
$$P_{D,K}: \mathcal A (Z_{H_D}(\A) G_D(F)\backslash G_D(\A))
\otimes \ol{\mathcal A_{cusp} (Z_K(\A)G_K(F)\backslash G_K(\A))}\rightarrow \C$$
by
$$P_{D,K}(f,\varphi)=
\int\limits_{Z_K(\A)G_K^0(F)\backslash G^0_K(\A)} f(g) \overline{\varphi(g)} \, dg.$$
The convergence of this period follows from the cuspidality of $\varphi$.

We investigate the non-vanishing of $P_{D,K}$
on the representation $\Pi\boxtimes \ol\pi$
whenever   $\Pi\in A^D_{E,\eta}$ and  $\pi$ is a cuspidal
representation of $G_K(\A)$ whose central character has trivial restriction
to  $\I_F$.


\begin{Thm}\label{global:restriction:thm}
Let  $\Pi\in A^D_{E,\eta}$ be
an automorphic representation and let
 $\pi$ be an  irreducible cuspidal representations of $G_K(\A).$
\begin{enumerate}
\item
If $K=E$ then $P_{D,K}$ vanishes on $\Pi\boxtimes \ol\pi$.
\item
If  $K\neq E$ then $P_{D,K}$  vanishes on $\Pi\boxtimes \ol\pi$  if and only if
 $\Hom_{G^0_K(\A)}(\Pi,\pi)=0.$
\end{enumerate}
\end{Thm}

\begin{proof}
(1) The first statement follows immediately from the  see-saw identity
and Remark \ref{vanishing:E=K}.

(2) Assume $K\neq E$. If $P_{D,K}\neq 0$ then
obviously $\Hom_{G^0_K(\A)}(\Pi^S_\eta\boxtimes \ol{\pi},\C)\neq 0$.
Let us prove the other direction.
If $\Hom_{G^0_K(\A)}(\Pi^S_\eta\boxtimes \ol{\pi},\C)\neq 0$
then for any finite place $v$ the representation $\pi_v$ is  dihedral
representation of $G_K(F_v)$
with respect to $L_v$ and some character.
 In particular, $\pi_v\simeq \pi_v\otimes \chi_{L_v/K_v}$ for any finite $v$.
Equivalently, $\pi$ and $\pi\otimes \chi_{L/K}$ are nearly equivalent
representations and hence by strong multiplicity one are isomorphic.
So, $\pi=\pi(\eta_L)$ is a global dihedral representation
of $G_K(\A)$.
By the main local theorem, $\eta_{L_v}$ matches $\eta_v$ for
any finite $v\in pl(F)$. Hence by Lemma \ref{local:global:1}
$\eta_L$ matches $\eta$.  Without loss of generality we may assume
that $\eta_L|_{\I_E}=\eta$.


By the see-saw identity, the non-vanishing of
$P_{D,K}$ on $\Pi\boxtimes \ol{\pi}$  is equivalent
to the  non-vanishing
of the integral
$$\int\limits_{Z_D(\A)H_D(F)\backslash H_D(\A)}
 E_{\eta_L}(\phi_K)(h) \,\ol{ E_{\eta}(\phi)(h)}\, dh$$
for some pure tensor products vectors
$\phi_K\in \tau_{\eta_L}^\emptyset\subset \Ind^{H_K(\A)}_{H_K^c(\A)}\eta_L $ and
$ \phi\in \tau_\eta^S\subset \Ind^{H_D(\A)}_{H_D^c(\A)}\eta.$

One has
$$\int\limits_{Z_D(\A)H_D(F)\backslash H_D(\A)}
E_{\eta_L}(\phi_K)(h) \ol{E_{\eta}(\phi)(h)}\, dh=$$
$$\integral{\mu_2}
\int\limits_{\I_F E^\times \backslash \I_E}
\sum_{\gamma_1 \in H_K(F)^c\backslash H_K(F)}\phi_K(\gamma_1xs)
\sum_{\gamma_2\in H_D(F)^c\backslash H_D(F)}\ol{\phi(\gamma_2xs)}\, dx\, ds$$

$$\integral{\mu_2}
\sum_{(\gamma_1, \gamma_2) }
\left(\int\limits_{ \I_F E^\times \backslash \I_E}
(\eta_L)^{\gamma_1}(x) \ol{\eta^{\gamma_2}(x)} dx
\right)
\phi_K(\gamma_1s) \ol{\phi(\gamma_2s)} \, ds.$$

Since $\gamma_1,\gamma_2$ is acting on $\I_L$ and $\I_E$
respectively by Galois action the inner integral vanishes unless
$(\eta_L)^{\gamma_1}|_{\I_E}=\eta^{\gamma_2}$
 in which case it is equal to
the measure of $\I_F E^\times \backslash \I_E$.

 %

Assume that $\eta\neq \eta^\sigma$.
Then only elements of the form
 $(\gamma,\gamma)$ contribute to the inner sum where
$\gamma\in H_D^c(F)\backslash H_D(F)$.

Thus the integral above equals
$$\integral{\mu_2} \sum_{\gamma\in \mu_2(F)}
\phi_K(\gamma s) \ol{ \phi(\gamma s)}\, ds=
\int\limits_{\mu_2(\A)}\phi_K(s) \ol{\phi(s)}\,ds.$$
Since $\Hom_{H_D(\A)}(\tau^\emptyset_{\eta_L},\tau_\eta^{S})\neq 0$
one  can always choose $\phi_K$ and $\phi$ such that the local integrals
equal $1$
and hence also the global integral  equals 1.

The case $\eta=\eta^\sigma$ is treated similarly. We omit the details.

\end{proof}
\section {Compatibility with Ichino-Ikeda refined conjecture}

\subsection{The conjecture for tempered parameters}
Gross and Prasad have conjectured that for tempered automorphic
representation $\Pi\boxtimes \ol\pi$ of $SO(V')\times SO(U')$
such that $\Hom_{SO(U')(\A)}(\Pi,\pi)\neq 0$,
the non-vanishing of the period $P_{V',U'}$ on
$\Pi\boxtimes \ol{\pi}$
is equivalent to the non-vanishing of $L(\Pi\boxtimes \ol{\pi},1/2)$.

Later Ichino and Ikeda [II] refined the conjecture expressing
$$\frac{|P_{V',U'}(f\boxtimes \ol{\varphi})|^2}{\|f\|^2\cdot\|\varphi\|^2}
\quad f\in \Pi, \varphi \in \pi$$
as a product of
\begin{equation}\label{II:conj}
\frac{L(\Pi\boxtimes \ol{\pi},s)}
{L(\Pi,Ad, s+1/2)L(\ol{\pi}, Ad, s+1/2)}|_{s=1/2}
\end{equation}
and a  finite number of certain local integrals, whose non-vanishing
is related to the non-vanishing of
 $\Hom_{SO(U')(F_v)}(\Pi_v,\pi_v)$ for $v\in \Omega$ for some finite set $\Omega$.

\subsection{The conjecture for non-tempered parameters}

There is a difficulty in extending the refined  conjecture
for the non-tempered representations since the local integrals
do not converge and hence a regularization is required.
For Saito-Kurokawa representations such a regularization
has been recently carried out in [Q].

We investigate the following weak version of Ichino-Ikeda conjecture
for the non-tempered representations.

\begin{Conj}
Let $\Psi_1\times \Psi_2$ be Arthur parameters of $SO(V)\times SO(U),$
and $\Phi_1\times \Phi_2$ be associated Langlands parameters.
Let  $\Pi\boxtimes \ol\pi$  be a cuspidal automorphic
representation of  $SO(V')\times SO(U')$ in the packet
 $A_{\Psi_1}\times A_{\Psi_2}$
such that $\Hom_{SO(U')(\A)}(\Pi,\pi)\neq 0.$
The non-vanishing of the period $P_{V',U'}$
on $\Pi\boxtimes \ol{\pi}$ is equivalent to the non-vanishing of

\begin{equation}\label{II:conj}
\frac{L(\Phi_1\times \Phi^\vee_2,s)}
{L(\Phi_1,Ad, s+1/2)L(\Phi^\vee_2, Ad, s+1/2)}|_{s=1/2}.
\end{equation}
\end{Conj}

The conjecture holds for Saito-Kurokawa packets as shown in [GG].
Let us show that it also holds for the packets of the type $\theta_{10}$.

The L-parameter associated to $\Psi_{E,\eta}$ equals
$\Phi_{E,\eta}|\cdot|^{1/2}\oplus \Phi_{E,\eta}|\cdot|^{-1/2}$.
The parameter $\Psi_2$ is associated to a representation $\pi$ of $G_K(\A)$
whose central character has trivial restriction to $\I_F$.

Thus, (\ref{II:conj}) equals

$$\frac{L(\Phi_{E,\eta}\times \Phi^\vee_{\pi},s+1/2)
L(\Phi_{E,\eta}\times \Phi^\vee_{\pi},s-1/2)}
{L(\Phi_{\pi},Ad,s+1/2)L(\Phi_{E,\eta},Ad,s+1/2)\zeta(s+1/2)L(\Phi_{E,\eta},Ad,s+3/2) L(\Phi_{E,\eta},Ad,s-1/2)}|_{s=1/2}.$$

Assume that this expression does not vanish.
The denominator has a simple pole at $s=1/2$ coming from the factor $\zeta(s+1/2)$.
Hence the numerator also must have a pole. The numerator is the product
of two triple L-functions whose analytic behavior was studied in [I].
In particular,  if $L(\pi(\eta)\times \ol{\pi},s)$ has a pole at $s=1$ then
$K\neq E$ and $\pi=\pi(\eta_L)$ with $\eta_L$ matches $\eta$.
In this case $L(\pi(\eta)\times \ol{\pi},s)$ also has a pole at $s=0$

Conversely, assume $P_{D,K}(\Pi\boxtimes \ol{\pi})\neq 0$.
Then $K\neq E$ and $\pi=\pi(\eta_L)$, where $\eta_L$ matches $\eta$.
In this case
$$L(\pi(\eta)\boxtimes \ol\pi(\eta_L),s)=
\zeta_E(s)
L_E(\eta^{-1} \sigma(\eta),s)L_L(\eta^{-1}_L \sigma(\eta_L),s)$$
and hence has a pole at $s=1$ and at $s=0$.
Thus the (\ref{II:conj}) does not vanish.

\end{document}